\documentclass{article}
\usepackage{amsmath,amssymb,amsthm,graphicx,epsfig,float,url,mathtools}
\usepackage[colorlinks=true,citecolor={ForestGreen},linkcolor={ForestGreen}]{hyperref}
\usepackage{pdfsync}
\usepackage[usenames, dvipsnames]{xcolor}
\usepackage{tikz}
\usepackage{subfig}
\usepackage{bbm}
\usepackage{stmaryrd}
\usepackage{mathrsfs}  
\usepackage{caption}
\usepackage{float}
\usepackage{esint}
\usepackage[english]{babel}
\usepackage{dsfont}
\usepackage{graphicx}
\usepackage{grffile}

\usepackage[shortlabels]{enumitem}

\usepackage[makeroom]{cancel}
\usetikzlibrary{patterns}
\usepackage{accents}

\newcommand{\R}{\textnormal{I\kern-0.21emR}}
\newcommand{\N}{\textnormal{I\kern-0.21emN}}

\newcommand{\U}{\mathcal{U}_{L,M}}

\def\cV{{\mathcal V}}
\def\M{{\mathcal M}}

\def\Th{{\Theta}}

\renewcommand{\geq}{\geqslant}
\renewcommand{\leq}{\leqslant}

\def\e{{\varepsilon}}

\newcommand{\Jt}{\widetilde{J}}
\newcommand{\cJt}{\widetilde{\mathcal J}}
\newcommand{\Vt}{\widetilde{V}}

\newcommand{\cVt}{\widetilde{\mathcal{V}}}
\newcommand{\Tht}{\widetilde{\Theta}}
\renewcommand{\th}{\theta}

\DeclareMathOperator{\supp}{\mathop{\rm supp}}

\allowdisplaybreaks

\def\YYint#1#2#3{{\setbox0=\hbox{$#1{#2#3}{\iint}$}
    \vcenter{\hbox{$#2#3$}}\kern-.51\wd0}}
 
\usepackage[dvipsnames]{xcolor}

\newtheorem{theoremno}{Theorem}

\newtheorem{theorem}{\textbf{Theorem}}[section]

\newtheorem{remark}[theorem]{\textbf{Remark}}

\newtheorem{lemma}[theorem]{\textbf{Lemma}}

\newtheorem{proposition}[theorem]{\textbf{Proposition}}

\newtheorem{definition}[theorem]{\textbf{Definition}}

\numberwithin{equation}{section}

\usepackage{xargs}
 \usepackage[colorinlistoftodos,textsize=small]{todonotes}
 \newcommandx{\unsure}[2][1=]{\todo[linecolor=red,backgroundcolor=red!25,bordercolor=red,#1]{#2}}
 \newcommandx{\change}[2][1=]{\todo[linecolor=blue,backgroundcolor=blue!25,bordercolor=blue,#1]{#2}}
 \newcommandx{\info}[2][1=]{\todo[linecolor=green,backgroundcolor=green!25,bordercolor=green,#1]{#2}}
 \newcommandx{\improvement}[2][1=]{\todo[linecolor=yellow,backgroundcolor=yellow!25,bordercolor=yellow,#1]{#2}}
 
  \newcommandx{\biblio}[2][1=]{\todo[linecolor=blue,backgroundcolor=magenta!25,bordercolor=blue,#1]{#2}}

\begin{document}
\title{The tragedy of the commons:
\\ A Mean-Field Game approach to the reversal of travelling waves}


\author{Ziad Kobeissi\footnote{Institut Louis Bachelier, Inria Paris, (\texttt{ziad.kobeissi@inria.fr})},    \\Idriss Mazari-Fouquer\footnote{CEREMADE, UMR CNRS 7534, Universit\'e Paris-Dauphine, Universit\'e PSL, Place du Mar\'echal De Lattre De Tassigny, 75775 Paris cedex 16, France, (\texttt{mazari@ceremade.dauphine.fr})},   \\Domenec Ruiz-Balet\footnote{Department of Mathematics, Imperial College, South Kensington, London, UK, (\texttt{d.ruiz-i-balet@imperial.ac.uk})}}

\maketitle

\begin{abstract}
The goal of this paper is to investigate an instance of the tragedy of the commons in spatially distributed harvesting games. The model we choose is that of a fishes' population that is governed by a parabolic bistable equation and that fishermen harvest. We assume that, when no fisherman is present, the fishes' population is invading (mathematically, there is an invading travelling front). Is it possible that fishermen, when acting selfishly, each in his or her own best interest, might lead to
a reversal of the travelling wave and, consequently, to an extinction of the global population? To answer this question, we model the behaviour of individual fishermen using a Mean Field Game approach, and we show that the answer is yes. We then show that, at least in some cases, if the fishermen coordinated instead of acting selfishly, each of them could make more benefit,  while still guaranteeing the survival of the population. Our study is illustrated by several numerical simulations.
\end{abstract}


\medskip

\noindent\textbf{AMS classification (MSC2020): }35C07, 35K10, 35K57, 35Q89, 49N80 .
\paragraph{Acknowledgement:} This work was started during a stay of D. Ruiz-Balet in CEREMADE, Paris Dauphine Universit\'e PSL, and D. Ruiz-Balet gratefully acknowledges the support of IRL Short-Term Exchange program through the grant STEG2201DR, which made this stay possible. D. Ruiz-Balet was supported by the UK Engineering and Physical Sciences Research Council (EPSRC) grant EP/T024429/1. I. Mazari-Fouquer was supported by the French ANR Project ANR-18-CE40-0013 - SHAPO on Shape Optimization.

\section{Introduction}
\subsection{Scope of the paper}

The preservation of biodiversity is one of the major scientific, political and economical challenges of our present time, as the drive to ensure the survival of a population is often at odds with economic goals \cite{BBC2,costello2012status,davies2012extinction,hamilton2001outport,pikitch2012risks,pinsky2011unexpected}. 
 A typical example of such a situation is the management of fisheries \cite{BBC50,BBCwaste,costello2012status,worm2012future,Worm_2006}, which is the central theme of this paper. Indeed, it is by now well-known that overfishing plays a great role in the collapse of biodiversity in oceans, which in turn leads to dramatic ecological consequences: for instance, the water quality and the recovery potential of ecosystems have been shown to be intimately tied to biodiversity \cite{Worm_2006}. This is part of the drive for the COP15 to include the restoration and preservation of at least 30\% of the ocean biodiversity in their objectives' list.

Our goal, in this paper, is to present a mathematical analysis of the ubiquitous \emph{tragedy of the commons} from the point of view of spatially distributed harvesting games, using a travelling wave approach. To clarify our terms, first recall that the tragedy of the commons \cite{hardin2009tragedy}, a term coined by W.F. Lloyd \cite{Lloyd_1833}, is a general principle of economics that reads:
\begin{center}
\emph{The action of selfish players, each playing so as to maximise his/her outcome from a common resources, will eventually lead to the depletion of this
resources.}\end{center} Our context (see \eqref{Eq:MFG} for a precise model) is thus the following: we consider a population of fishes that evolves according to a standard reaction-diffusion equation of bistable type, and a large population of $N\gg 1$ fishermen, each acting so as to maximise his/her fishing output. Bistable equations typically admit an extinction state (0 for simplicity), an invasion state (1 for notational convenience) and an intermediate, unstable equilibrium state. Our goal is to show that, when acting in an uncoordinated fashion, the resulting action of fishermen will lead to the extinction of the fishes' population.  While there are many ways to quantify the extinction or survival of a species in reaction-diffusion models \cite{CantrellCosner}, we adopt the point of view of \emph{invading travelling waves}. Since their introduction in the seminal works of Fisher \cite{Fisher} and, independently, of Kolmogorov, Petrovskii and Piskounov \cite{KPP}, travelling waves have proved to be a very convenient framework for the qualitative understanding of population dynamics.

Roughly speaking, if we work in $\R$, the typical population dynamics equation is of the form
\[ \partial_t u-\partial^2_{xx}u=f(u)-mu.\] Here, $u$ is the population density, $f$ is a specific reaction term, and $m$ is the density of fishermen. The term $-mu$ corresponds to the harvesting of fishes. When $m\equiv 0$ (\emph{i.e.} in the absence of fishermen) a \emph{travelling front} at speed $c\in \R$ is a solution $u$ that writes 
\[ u(t,x)=U(x-ct)\text{ with }\begin{cases}-cU'-U''=f(U)&\text{ in }\R\,, 
\\ U(-\infty)=0\,, U(+\infty)=1.\end{cases}\]
If a travelling wave (TW) solution exists with a \emph{negative} speed $c<0$, it means that the population will typically survive in this scenario (``1 is invading 0'') and we thus dub it \emph{invading travelling wave}. If, on the other hand, there exists a TW solution with a \emph{positive} speed $c>0$, this means that the population will go extinct (``0 is extinguishing 1'') and we refer to it as an \emph{extinction travelling wave}. travelling waves have been the subject of such an intense research activity that it would be pointless to try and give an exhaustive bibliography. We refer to section~\ref{subsec:related_works} for more details. 
 
Depending on $f$, an infinity or a single travelling wave solution might exist. Since we want to exemplify the way fishing can lead to extinction, we focus in this article on the case of a bistable non-linearity (typically $f:[0;1]\ni u\mapsto u(u-\eta)(1-u)$); in this case, it is well known \cite{FifeMcLeod} that there exists a \emph{unique} travelling wave that can be invading or extinguishing. 
Our question then becomes:
\begin{center}{
Assume that in the absence of fishermen there only exists one  travelling wave, which in addition is invading. Is it possible that adding fishermen will generate an extinction travelling wave?}
\end{center}
This question was answered (among others) in \cite{Bressan} (see section~\ref{subsec:related_works}) but is not fully satisfactory for our needs; for indeed, it is not clear whether or not a strategy leading to the depletion of the fishes' population is optimal from a fisherman's perspective. Our real central question is then 
\begin{center}\emph{
Assume that in the absence of fishermen there only exists one travelling wave, which in addition is invading. Is it possible that adding fishermen that act in their own best interest will generate an extinction travelling wave?}
\end{center}
We provide, in a variety of situations, a \emph{positive} answer to this question.

Keeping in mind that we use reaction-diffusion equations and travelling waves to model population dynamics, we need to settle on a paradigm to implement the behaviour of the fishermen. In the present paper, we choose a Mean-Field Game (MFG) approach which has proved a very efficient tool in the modelling of agents motivated by their own self interest. 

Mean-Field Game models have, since their introduction by Lasry \& Lions
\cite{MR2269875,MR2271747,MR2295621}
and, independently by Huang, Caines \& Malhamé
\cite{MR2352434,MR2346927}
been at the center of various fields of applied mathematics and have been a major subject of investigation from the optimisation and PDE communities. While our paper is not the first one to tackle the interplay between MFG and travelling waves \cite{Burger_2017,Papanicolaou_2021,Porretta_2022,Qin_2019}, both our problem and approach are very different from theirs; we refer to section~\ref{subsec:related_works}. In particular, this article is, to the best of our knowledge, the first to blend travelling waves, Mean-Field Games and the tragedy of the commons. 

Our conclusions account for the fact that, when fishermen are only motivated by their own selfish interest, their cumulative action might lead to the extinction of the fishes' population. Let us highlight that we also cover an interesting aspect of this model: when fishermen \emph{coordinate}, not only does their cumulative action not extinguish the fishes' population, but all individual benefits are actually higher than when they compete. This further reinforces our point that our ``competition leads to extinction" type of results can be seen as illustrating the tragedy of the commons. Let us also note that although we mostly use, in the course of this article, fisheries-related vocabulary (\emph{e.g.} fishermen), the paradigm we exemplify remains valid in a variety of other harvesting settings (\emph{e.g.} logging, deforestation).

Finally, we want to emphasise that, while we describe and analyse a wide range of situations, several crucial questions remain open. Indeed, our results point at the \emph{existence} of extinction travelling waves. In spatial ecology, the existence of such fronts is usually complemented with a detailed analysis of their \emph{stability} \cite{FifeMcLeod}. In other words: is it true that for ``reasonable" initial conditions, the solutions of the equation ``look like" this front for large times? Answering this question would be an important step in the qualitative analysis of fishing problems, but is at the moment out of reach. We refer to the conclusion for some open problems.

\subsection{Mathematical set-up}

\paragraph{Bistable equations: setting and classical results in the absence of fishermen}
We begin with the basic definition:
\begin{definition}[Bistable non-linearity]\label{DE:BNL}
    A function $f$ is called \emph{bistable} if it satisfies:
\begin{enumerate}
\item $f\in \mathscr C^1([0,1];\R)$,
\item $f$ has exactly three roots in $[0;1]$: $0$, $1$ and some $\eta\in (0;1)$,
\item $f'(0)\,, f'(1)<0$ and $f'(\eta)>0$.
\end{enumerate}
\end{definition}Bistable nonlinearities are of particular importance to model the Allee effect \cite{Perthame}.
The most standard example of such a bistable non-linearity is the following: fix a parameter $\eta\in (0;1)$ and define 
\begin{equation*}\label{Eq:Al} f:[0;1] \ni u\mapsto u(u-\eta)(1-u).\end{equation*} While we will work under mild assumptions on $f$, non-linearities of the latter type are important to keep in mind.

Consider, for a bistable non-linearity $f$, the associated bistable equation
\begin{equation}\label{Eq:Main2} \partial_t u-\partial^2_{xx} u
    =f(u)\text{ in }[0;\infty)\times \R,\end{equation} where the initial condition is willingly omitted. We have the following definition:
    
\begin{definition}[Travelling wave solution]\label{DE:TW}
A travelling wave solution of \eqref{Eq:Main2} is a couple $(U,c)$, with $c\in \R$ and where $U$ satisfies
\begin{equation}\label{TW}\begin{cases}
-U''-cU'=f(U)&\text{ in }\R\,, 
\\ U(\infty)=0\,, U(+\infty)=1.\end{cases}\end{equation}
If $c>0$, the wave is dubbed ``extinction travelling wave'' or ``extinction front" while, if $c<0$, the wave is dubbed ``invading travelling wave'' or ``invading front".
\end{definition}
If $(U,c)$ meets the condition of Definition~\ref{DE:TW}, then $u:(t,x)\mapsto U(x-ct)$ is a solution of \eqref{Eq:Main2}.
Starting with \cite{Fisher,KPP}, travelling waves have been a central feature of mathematical biology. The seminal work of Fife \& MacLeod \cite{FifeMcLeod} reviews various results related to the existence of  travelling waves solutions and investigates their stability properties. Roughly speaking, solutions of \eqref{Eq:Main2} with an initial condition $0\leq u_0\leq 1$, $u_0(-\infty)<\eta\,, u_0(\infty)>\eta$ are, in the limit $t\to \infty$, similar to travelling fronts \cite[Theorem 3.1]{FifeMcLeod}.  One of the foundational results in the field is the following theorem:
\begin{theoremno}\cite[Theorem 3.2]{FifeMcLeod}
    \label{Th:WithoutFisherman}
Let $f$ be a bistable non-linearity, as defined in Definition~\ref{DE:BNL}, and assume that 
\begin{equation}
    \label{Eq:invasive_BNL}
    \int_0^1 f(s)ds>0.
\end{equation}
Then there exists (up to translation) a unique travelling wave solution $(U^*,c^*_{TW,f})$ of \eqref{Eq:Main2}.
Furthermore,
this solution is an invasion front
\emph{i.e.} $c^*_{TW,f}<0$. Finally, for any initial condition $u_0$ such that 
\[ \lim_{s\to -\infty} u_0\in[0;\eta)\,, \lim_{s \to +\infty} u_0\in (\eta;1],\] there exists $x_0\in \R\,, w>0$ such that 
\[ \sup_{x \in \R} | u(t,x)-U^*(x-ct-x_0)|=\underset{t\to \infty}O (e^{-wt}).\]

\end{theoremno}
As outlined in the introduction, our purpose is to understand the influence of fishing on population dynamics through the lens of travelling fronts, using a mean field game approach.

\paragraph{Mean Field Game} We consider a fixed bistable nonlinearity $f$.
For a given density of fishermen $m=m(t,x)$, the fishes' population solves
\begin{equation*}\label{Eq:Main3}
\partial_t \theta-\partial^2_{xx}\theta=f(\theta)-m\theta.
\end{equation*}
To take into account the fact that fishermen follow an ``optimal'' strategy, we follow the seminal \cite{MR2295621,Lions_video} and we assume that the density $m=m(t,x)$ is the limit of a population of $N\gg1$ individual fishermen (we refer to \cite{Cardaliaguet_2019} for a study of the convergence of the model as $N\to \infty$). Each of them tries to optimise her/his output by controlling her/his displacement; this is modelled by saying that the representative fisherman is described \emph{via} the ODE
\[ x'=\alpha\] where $x$ is the position of the fisherman and $\alpha$ is the control. The resulting density of players follows the continuity equation 
\begin{equation}\label{Eq:Continuity} \partial_t m+\partial_x(\alpha m)=0.\end{equation}
Starting in an initial position $x_0$, the player seeks to optimise his/her fishing output. This output depends on:
\begin{enumerate}
\item A discount factor $\lambda>0$,
\item The average selling price of the fish, which we normalise to be equal to 1,
{\item The cost of the control, which is enforced through a Lagrangian $L$.}
\end{enumerate}
Overall, for a fixed fishes' population $\theta$,
an individual fisherman starting at position $x_0$ solves the optimisation problem
\begin{equation}\label{Eq:OCPlayer} \sup_{\alpha \in L^\infty(\R)} J(x_0,\alpha,\th):=\int_0^\infty e^{-\lambda t}\left(\theta(t,x_{\alpha}(t))-L(\alpha)\right)dt\text{ with }\begin{cases}x_\alpha'(t)=\alpha(t)\,,\\ x_\alpha(0)=x_0.\end{cases}\end{equation} 
When there is no ambiguity in the choice of $\th$,
we allow ourselves to omit the dependence in $\th$ is $J$
and we simply write $J(x_0,\alpha)$.

The assumptions on the Lagrangian $L$ are standard:
\begin{equation}\tag{$\bold{H}_L$}\label{Hyp:L}
\begin{cases}
\text{$L$ is superlinear:  } \underset{|\alpha| \to \infty}\lim\frac{L(\alpha)}{1+|\alpha|}=\infty\\ \text{ $L$ is $\mathscr C^1$ in $\alpha$, and strictly convex in $\alpha\in \R$},
\\ \text{$0$ is the unique global minimum of $L$ (up to an additive constant we take $L(0)=0$)}
\end{cases}\end{equation}A typical example of a Lagrangian $L$ satisfying \eqref{Hyp:L} is
$L(\alpha)=|\alpha|^q\,, q>1$.
The value function $V$ is defined as
\begin{equation}\label{Eq:DefV}V:(t_0,x_0)\mapsto \sup_{\alpha\in L^\infty((t_0;\infty))}
    \left(\int_{t_0}^\infty e^{-\lambda (t-t_0)}\left(\theta(t,x_{\alpha}(t))
            -L(\alpha)\right)dt\right)\text{ with }
    \begin{cases}x_\alpha'(t)=\alpha(t)\,,\\ x_\alpha(t_0)=x_0.\end{cases} \end{equation}
Under regularity assumptions on the optimal control a standard application of the Bellman dynamic programming principle
\cite[Chapter I, Section 2]{Bardi_1997} shows that
the value function 
$V$
is the unique viscosity solution of the Hamilton-Jacobi-Bellman (HJB) equation
\begin{equation*}\label{Eq:HJBCompet}\begin{cases}
\lambda V-\partial_t V-H(\partial_xV)=\theta\text{ in }(0;\infty)\times \R\,,\\ V(+\infty,\cdot)\equiv 0,\end{cases}\end{equation*}
where the Hamiltonian $H$ is defined as the Legendre transform of $L$:
\[ H(p):=\sup_{\alpha \in \R}\left( p\alpha-L(\alpha)\right).\]
Provided \eqref{Hyp:L} is satisfied,
$H$ has a locally bounded derivative.
In this case, there exists a unique optimal control
reaching the maximum in the definition of $V$ (Eq. \eqref{Eq:DefV}),
moreover it admits a feedback form
given by 
\begin{equation}
    \label{Eq:opt_cont}
    \alpha^*(x)=H'(\partial_xV(x)).
\end{equation}
The fishermen' density $m$  is then a distributional solution of the continuity equation
\begin{equation}\label{Eq:FPKCompet}
\begin{cases}
\partial_tm+\partial_x\left(H'(\partial_x V) m\right)=0\text{ in }(0;\infty)\times \R\,, \\m(0,\cdot)=m_0.\end{cases}
\end{equation}
\begin{remark}
At this stage, $m$ could very well be a measure in space. However, we will show (Lemma~\ref{Le:NoDirac}) that for the travelling wave solutions we look for, the fishermen's density $m$ should not have atoms. Our proofs will be constructive, and this fact is one of the motivations between building possible $m$ as an $L^1$ function.
\end{remark}
With some initial conditions $\theta_0$ and $m_0$,
adding a terminal time horizon $T>0$ and a terminal cost
$V_T:\R\rightarrow\R$ for convenience,
we are left with the MFG system:
\begin{equation}\label{Eq:MFG}
\begin{cases}
\lambda V-\partial_t V-H(\partial_xV)=\theta&\text{ in }(0;T)\times \R\,,
\\ V(T,\cdot)=V_T,
\\\partial_tm+\partial_x\left(H'(\partial_x V) m\right)=0&\text{ in }(0;T)\times \R\,, \\m(0,\cdot)=m_0,
\\ \partial_t \theta- \partial_{xx}^2\theta=f(\theta)-m\theta&\text{ in }(0;T)\times \R\,, 
\\ \theta(0,\cdot)=\theta_0.
\end{cases}
\end{equation}

\subsection{Reversed travelling waves: first discussion}
\paragraph{Definition of a reversed MFG travelling wave}
We now come to the heart of the matter by defining the notion of reversed travelling wave for \eqref{Eq:MFG}. Throughout we work with a bistable nonlinearity $f$ that satisfies \eqref{Eq:invasive_BNL}.
\begin{definition}[Reversed MFG travelling wave]\label{De:ReversedMFGTW}
    We say that 
    $(\lambda,c,\Th,\mathcal M,\mathcal V)$
    is a \emph{reversed MFG travelling wave} solution of \eqref{Eq:MFG} if    $c>0$,
    $\displaystyle{\lim_{s\to-\infty}}\Th(s)=0$,
    $\displaystyle{\lim_{s\to+\infty}}\Th(s)=1$, $\Th>0$ in $\R$, if $(\Th,\M)$ satisfies
    \[ -\Th''-c\Th'=f(\Th)-\M\Th \text{ in }\R\]
    and if the constant control $\overline\alpha\equiv c$ is optimal in \eqref{Eq:DefV}
    where $\theta:(t,x)\mapsto \Th(x-ct)$, for any $x_0\in \supp(\M)$.
\end{definition}
In particular, this implies that, defining
    \[ V:(t,x)\mapsto \cV(x-ct)\,, m:(t,x)\mapsto \M(x-ct),\]
    the triplet $(\theta,m,V)$ is a solution of \eqref{Eq:MFG} with
    $m_0=\M\,, \theta_0=\Th\,, V_T=\cV$.

The wording ``reversed'' comes from the positivity condition $c>0$,
which means that we are in the presence of an extinction front;
this is to be contrasted with the results recalled in  Theorem~\ref{Th:WithoutFisherman}.

If
$(c,\Th,\mathcal M,\mathcal V)$ is a reversed MFG travelling wave, the change of variables $s=x-ct$ shows that $(\Th,\M,\mathcal V)$ solves the stationary system
\begin{equation}
    \label{Eq:MFG_s}
    \left\{
    \begin{aligned}
        &\lambda\cV(s)
        +c\cV'(s)
        -H(\cV'(s))
        =
        \Th(s)
        \text{ on }\R,
        \\
        &-\Th''(s)-c\Th'(s)
        =
        f(\Th(s))-\M(s)\Th(s)
        \text{ on }\R,
        \\
        &H'(\cV'(s))
        =
        c
        \text{ in supp}(\M),
    \end{aligned}
    \right..
\end{equation}
The variable $s$
is called the
\emph{similarity variable}.
Observe that the continuity equation in \eqref{Eq:MFG} has been replaced by the condition that the optimal control of a representative fisherman (that is, a solution of \eqref{Eq:opt_cont}) is equal to $c$, the velocity of the travelling wave. Moreover, $\cV$ is indeed the value function of the control problem expressed in similarity variables as
\begin{equation}\label{Eq:DefVSV}
    \begin{aligned}
        &\cV(s_0)
        =
        \sup_{\alpha\in L^\infty(\R)}
        \mathcal{J}(\alpha,s_0,\Th)\text{ where}
        \\
        &\mathcal{J}(\alpha,s_0,\Th)
        =
        \left(\int_{0}^\infty e^{-\lambda t}\left(\Th(s_{\alpha}(t))
                -L(\alpha)\right)dt\right)
        \text{ with }
        \begin{cases}
            s_\alpha'(t)=\alpha(s(t))-c\,,
            \\
            s_\alpha(t_0)=s_0,\end{cases}
    \end{aligned}
\end{equation}
where the supremum is directly taken
over the control function in feedback form (as a consequence of \eqref{Eq:opt_cont}).
When there is no ambiguity in the choice of $\Th$,
we allow ourselves to omit the dependence in $\Th$ is $J$
and we simply write $\mathcal{J}(x_0,\alpha)$.

The next paragraphs give some \emph{a priori} restrictions on reversed MFG travelling waves; since our approach in the rest of the paper is constructive, this helps us justify our forthcoming constructions.

\paragraph{Preliminary comments on reversed MFG travelling waves}
\begin{lemma}\label{Le:NoDirac}
    Assume that $L$ satisfies \eqref{Hyp:L}.
    If $(\lambda,c,\Th,\M,\cV)$ is
    a solution to \eqref{Eq:MFG_s},
    $\M$ does not have atoms. 
\end{lemma}

\begin{remark}[Comparison with \cite{Bressan}]
It is apparent with this result that the MFG setting under consideration here differs significantly from the situation considered in \cite{Bressan}, where the goal is to find a harvesting strategy that extinguishes the population; indeed, it is shown in \cite{Bressan} that Dirac-type fishermen density $m$ are in some sense optimal (among controls leading to reversed travelling waves) to kill off the fishes, in sharp contrast with the conclusion of Lemma~\ref{Le:NoDirac}.
\end{remark}

\begin{remark}
That $\M$ does not have atoms does not necessarily mean that $\M$ is an $L^1$ function; however, Lemma~\ref{Le:NoDirac} is a strong motivation for us to look for $\M$ as an $L^1$ function (or $L^1_{\mathrm{loc}}$ when dealing with the periodic case).
\end{remark}

\begin{proof}[Proof of Lemma~\ref{Le:NoDirac}]
    Assume by contradiction
    that $\M$ has
    an atom at $s_0\in\R$: $\M(\{s_0\})>0$.
    From \eqref{Eq:MFG_s} it follows that $\Th'$ has a jump at $s_0$:  \begin{equation*}
        \Th'(s_0^-)
        <
        \Th'(s_0^+).
    \end{equation*}
    By definition, for any $s\in \supp(\M)$, $\overline\alpha \equiv c$
    is optimal in \eqref{Eq:DefVSV}.
    Let  $h_+\equiv 1\in L^\infty(\R)$ and
    $\dot{\mathcal{J}}(\overline\alpha,s_0)[h]$ 
    be the first order Gateaux derivative of
    $\alpha\mapsto \mathcal{J}(\alpha,s_0)$ at $\overline \alpha$ in the direction $h_+$. By straightforward computations and from the optimality of $\overline \alpha$ we obtain
    \begin{equation*}\label{Eq:DJ}
    \dot{\mathcal{J}}(\overline \alpha,s_0)[h_+]=\int_0^\infty e^{-\lambda t}\left(t\Th'(s_0^+)-L'(c)\right)dt=\lambda^{-2}
        (\Th'(s_0^+)
        -\lambda L'(c))\leq 0.
    \end{equation*} 
    Similarly, the perturbation in the direction $h_-\equiv -1$ yields 
     \begin{equation*}\label{Eq:DJ2}
\dot{\mathcal{J}}(\overline\alpha,s_0)[h_-]
        =
        \int_0^{\infty}
        e^{-\lambda t}
        (-t\Th'(s_0^-)
        +L'(c))\,dt
        =
        -\lambda^{-2}
        (\Th'(s_0^-)
        -\lambda L'(c))\leq 0.
    \end{equation*} Summing the two latte inequalities yields
    \[\Th'(s_0^+)-\Th'(s_0^-)\leq 0,\] a contradiction.
\end{proof}
We now move on to a second comment, this one related to to the link between the support of $\M$ and the (non-)monotonicity of $\Th$, as well as to the behaviour of $\Th$ in $\mathrm{supp}(\M)$. Indeed, several of our theorems are devoted to the construction of travelling waves that feature a monotonous $\Th$, while the case of periodic or non-monotonous reversed MFG travelling waves is discussed in section~\ref{subsec:MR_periodic}. We now show the following lemma:
\begin{lemma}\label{Le:MonotComp}Assume $\M$ is $L^1$ and that $(\lambda,c,\Th,\M,\cV)$ is a reversed MFG travelling wave. $\Th$ is linear in the support of $\M$, with $\Th'=\lambda L'(c)$.
Moreover, if $\Th$ is monotone increasing, $\M$ must be compactly supported.
\end{lemma}

\begin{proof}[Proof of Lemma~\ref{Le:MonotComp}]
 Under assumption \eqref{Hyp:L}, the condition $H'(\cV')\equiv c$ in $\mathrm{supp}(\M)$ rewrites 
\[ \cV'\equiv L'(c)\text{ in }\mathrm{supp}(\M).\] Plugging this in the (HJB) equation in \eqref{Eq:MFG_s}, we deduce that $\Th$ is linear in $\mathrm{supp}(\M)$.  In particular, if $\Th$ is monotonic, since $\Th(s)\underset{s \to \infty}\to 1$, $\M$ is compactly supported.
\end{proof}
In the case of compactly supported $\M$, we will often assume (up to a translation) that $\min(\mathrm{supp}(\M))=0$ and, in many cases, we will even look for $\M$ with a connected support.

Finally, we provide an \emph{a priori} estimate on the maximal speed allowed.
\begin{lemma}
    \label{Rk:max_speed}
If $(\lambda,c,\Th,\M,\cV)$ is a reversed MFG travelling wave, then $c$ must satisfy
\[\forall s_0\in \mathrm{supp}(\M)\,,  L(c)\leq \Th(0).\]
    
    \end{lemma}
    \begin{proof}[Proof of Lemma~\ref{Rk:max_speed}]
    Since $\overline \alpha=c$ is optimal for any $s_0\in \mathrm{supp}(\M)$ we have, 
    \[ \mathcal{J}(0,s_0)\leq
        \mathcal{J}(\overline \alpha,s_0)\text{ for any $s_0\in \mathrm{supp}(\M)$}.\] As $L(0)=0$ this yields
    \[ 0<\int_0^\infty e^{-\lambda t}\Th(s_0-ct)dt=
        \mathcal{J}(0,s_0)\leq \mathcal{J}(\overline \alpha,s_0)=\frac{\Th(s_0)-L(c)}\lambda\] whence the conclusion.
    \end{proof}   This leads to defining the maximal speed associated with a Lagrangian $L$. To define it, some background on the phase-portrait analysis of bistable equations is required. While we refer for details to section~\ref{sec:construct_Th} we indicate that, for any $c\geq 0$, there exists a solution $\Gamma_{0,c}$ of the differential equation $-\Gamma_{0,c}''-c\Gamma_{0,c}'=f(\Gamma_{0,c})$ in $\R$, with $\Gamma_{0,c}(-\infty)=\Gamma_{0,c}'(-\infty)=0$ with $\Gamma_{0,c}>0$ in $\R$ (in
    the dynamical system terminology, $\Gamma_{0,c}$ corresponds to the unstable manifold associated with the equilibrium $(U=0\,, U'=0)$). For any  $L$ satisfying \eqref{Hyp:L}, we define 
    \begin{equation}\label{Eq:Defcmax} c_{\mathrm{max}}(L):=\sup\{c>0\,, L(c)<\Vert\Gamma_{0,c}\Vert_{L^\infty(\R)}\}.\end{equation} 
    
\section{Main results}
\label{sec:MR}
We have divided this section in three main parts. The first one, section~\ref{subsec:MR_monotonous}, deals with the existence of monotonous reversed MFG travelling waves. The second one, section~\ref{subsec:coordination}, is focused on the tragedy of the commons: we show in certain cases that, when coordinating, each fisherman can actually obtain a higher fishing outcome, while ensuring that the fishes' population survives. The third one, section~\ref{Se:Extensions} contains a shorter discussion of periodic and non-monotonous travelling waves.

\paragraph{Standing assumptions}Throughout, $f$ is a bistable function (in the sense of Definition~\ref{DE:BNL}) that satisfies \eqref{Eq:invasive_BNL}. Furthermore, we require that, if $\eta \in (0;1)$ is the third root of $f$, we ask that 
\begin{equation}\label{Hyp:f2}f\text{ has a unique critical point in $[0;\eta]$}.\end{equation} 
Unless stated otherwise, $L$ satisfies \eqref{Hyp:L}.

\subsection{Existence of  monotonous reversed MFG travelling waves}
\label{subsec:MR_monotonous}
Here we insist on the monotonicity of $\Th$. We say that a reversed MFG travelling wave $(c,\lambda,\Th,\M,\cV)$ is monotonous when $\Th$ is monotone increasing.

\subsubsection{Existence of particular quadruplets $(\lambda,c,\Th,\M)$} We isolate the following construction result.

\begin{proposition}
    \label{Prop:exist_th_TW}
There exists a map $c_0:(0;\infty)\to \R_+^*$ such that:
    \begin{enumerate}[i)]
        \item
    for all $c\in(0,c_0(\lambda)]$, there exists a unique couple $(\Th,\M)$ such that
    \begin{itemize}
        \item
            $\M\in L^1(\R;[0,\infty))$ and $\supp(\M)=[0,s_1]$ for some $s_1>0$,
        \item
            $\Th$ is piecewise $\mathscr C^2$ and increasing from $0=\displaystyle{\lim_{s\to-\infty}}\Th(s)$  to $1=\displaystyle{\lim_{s\to\infty}}\Th(s)$,
        \item
            $\Th$ satisfies  $-c\Th'-\mu\Th''=f(\Th) - \M\Th$ on $\R$,
        \item $\Th'\equiv \lambda L'(c)$ on $\supp(\M)$,
          \end{itemize}
    \item
        for $c> c_0(\lambda)$, no such couple $(\Th,\M)$ exists,
    \item
        $c_0$
        is continuous
        and decreasing, with $\lim_{\lambda \to 0^+}c_0(\lambda)=\infty$, $\lim_{\lambda \to \infty}c_0(\lambda)=0$.
    \end{enumerate}
\end{proposition}
We refer to Section~\ref{sec:construct_Th} for the proof of this proposition. The map $c_0$ is characterised in Eq. \eqref{Eq:c0} below.

The upcoming theorems all rely on the fine study of the particular triplets $(c,\Th,\M)$ given by Proposition~\ref{Prop:exist_th_TW}, and boil down to proving that $\overline \alpha=c$ is optimal in \eqref{Eq:DefVSV} up to certain conditions on $(c,\lambda)$. 

We represent, in Fig.~\ref{fig:sol_monotonous}, the way this triplet can be constructed. Fixing $c>0$, we represent on the right hand side the phase portrait of the ODE $-\Th''-c\Th'=f(\Th)$. The idea, in parts similar to \cite{Bressan}, is to first follow the green curve which corresponds to the unstable manifold of the equilibrium $(\Th=0,\Th'=0)$, and to start acting at a definite moment to reach the stable manifold associated with the equilibrium $(\Th=1,\Th'=0)$. We refer to section~\ref{sec:construct_Th} for details on this phase portrait. 
\begin{figure}[h]
\begin{center}
\includegraphics[scale=0.12]{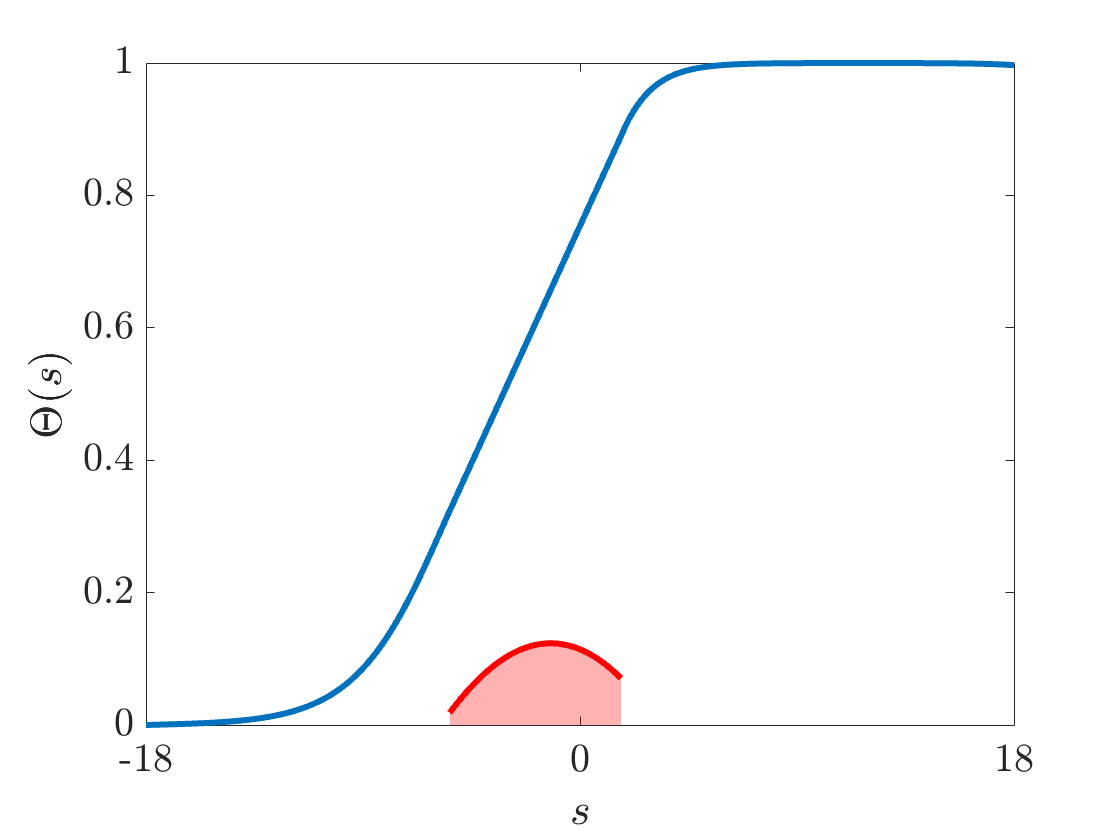}
\includegraphics[scale=0.12]{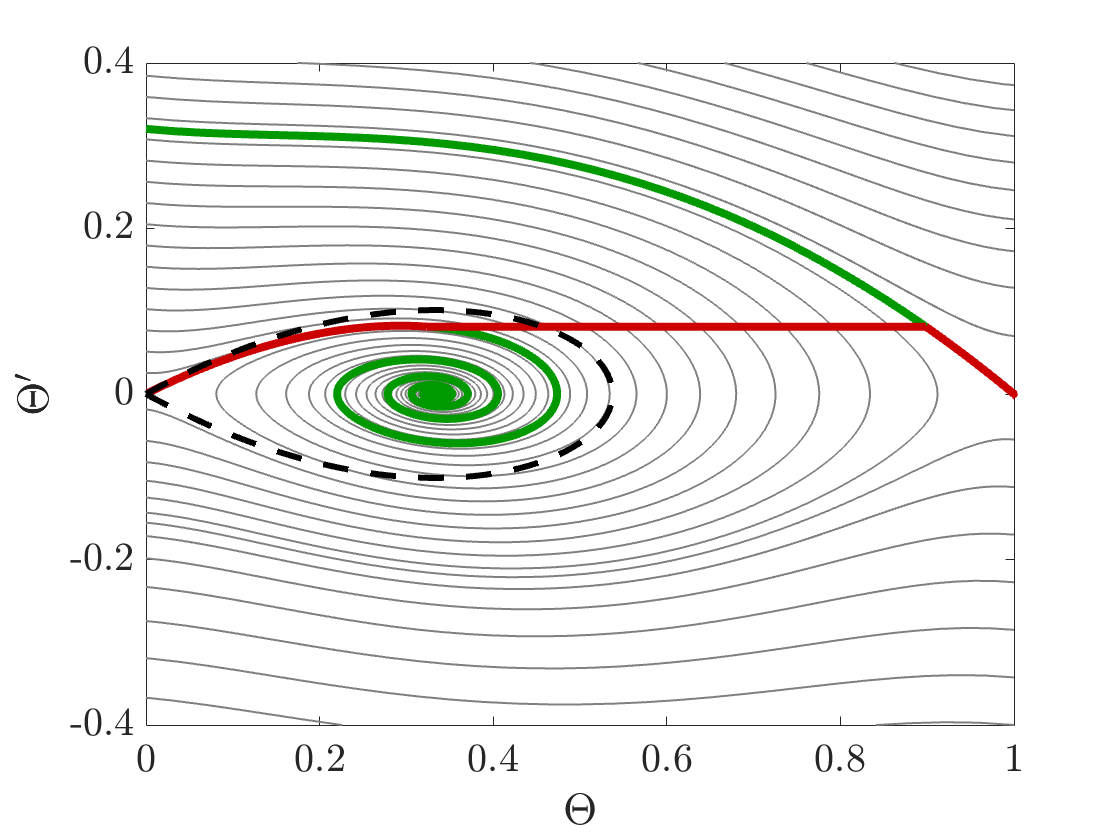}
\includegraphics[scale=0.12]{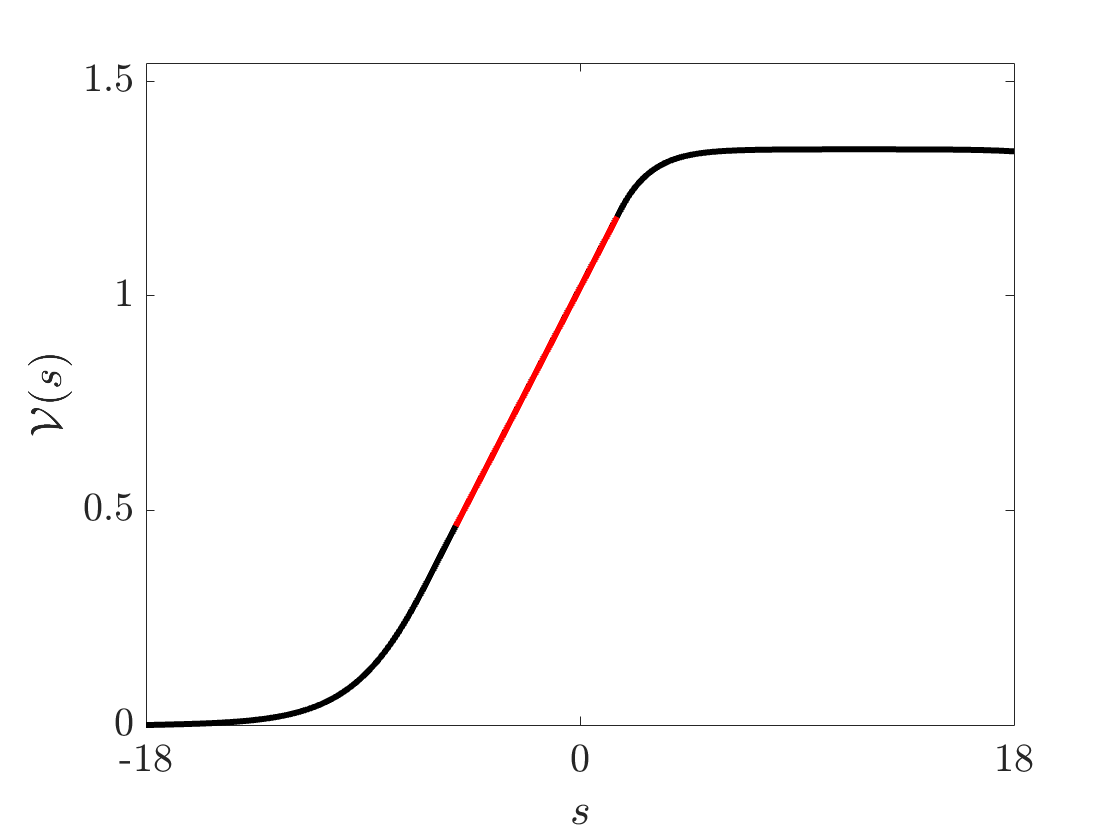}\end{center}
\caption{
    \label{fig:sol_monotonous}
   Reversed travelling wave for $0<c<2\sqrt{f'(\theta)}$. The travelling wave profile $\Th$ in blue and in the control $\M$ that generates it in red (left); Construction of this reversed travelling wave in the phase-plane; the stable manifold of $(1,0)$ and the unstable manifold of $(0,0)$ are in green. The  dashed black line represents the invariant. The reversed travelling wave trajectory is in red (center); the value function $\lambda=0.79$ the red segment where the support of $\mathcal{M}$ is located (right).
}
\end{figure}

\subsubsection{A first existence result for reversed MFG travelling waves}
Our first theorem is a general existence result:
\begin{theorem}\label{Th:Main}
For any $\lambda>0$,
    there exists
    $\e_0(\lambda)\in(0,c_0(\lambda))$
    such that,
    for  any $c\in(0,\e_0(\lambda))$
    there exists a unique
    monotonous
    reversed MFG travelling wave with velocity $c$.\end{theorem}
This theorem is proved in section~\ref{subsec:proof_Th_Main}. While a fairly general result, its main drawbacks are that, first, the map $\e_0$ is not easily computed and, second, that this statement provides a sufficient condition only. We refer to Fig.~\ref{fig:sol_monotonous}
for a graphical illustration of the profiles $(\Th,\M)$.

\subsubsection{Finer tuning of the speed of the travelling wave}
The next two theorems address the question of the speed $c$. From Theorem~\ref{Th:Main} and Lemma~\ref{Rk:max_speed} we know that $(\lambda,c)$ must satisfy $c<\min (c_{\max},c_0(\lambda))$.

In the following, we present
two other existence results,
proving that these necessary 
conditions are sharp in some
regimes.
\begin{theorem}
    \label{Th:small_lambda}
    For any $c\in(0,c_{\max})$,
    there exists $\lambda$ small enough
    such that there exists
    a monotonous
    reversed MFG travelling wave
    (in the sense of Definition~\ref{De:ReversedMFGTW})
    with velocity~$c$ and discount factor $\lambda$.
\end{theorem}
For the proof of this theorem, we refer to section~\ref{subsec:proof_small_lambda}. Theorem~\ref{Th:small_lambda} shows that the necessary condition $c<c_{\max}$ is almost sharp in the asymptotic regime $\lambda \to 0$. Observe that while it may at first sight seem like an innocuous corollary of Theorem~\ref{Th:Main}, the main difficulty is that Theorem~\ref{Th:Main} does not provide an asymptotic behaviour of $\e_0(\lambda)$ as $\lambda\to 0$.

The following theorem
focuses on the case
of strongly convex
Lagrangians.
In this case, and for big enough discount factors $\lambda$, we prove that the necessary condition $c\leq c_0(\lambda)$ is sharp. In other words, this is a regime where a reversed travelling wave is actually a reversed MFG travelling wave.
\begin{theorem}\label{Th:Quadratique}
    
    Assume that $L$ is $\mathscr C^2$
    and $L''\geq \underline D$
    for some positive constant $\underline D$.
    There exists $\lambda_0(\underline D)>0$
    such that, for any $(c,\lambda)$ satisfying
    \begin{equation*}
        \lambda\geq\lambda_0(\underline D)
        \qquad
        \text{ and }
        \qquad
        0<c\leq c_0(\lambda),
    \end{equation*}
    there exists a unique monotonous 
    reversed MFG travelling wave
    with velocity $c$ and discount factor~$\lambda$.
\end{theorem}The proof is given in section~\ref{sec:strongly_convex} and is very different from that of the previous theorems.

\subsection{Coordination \& the tragedy of the commons}
\label{subsec:coordination}
We chose the terminology \emph{coordination}
    to echo the wording \emph{uncoordinated}
    in the quote of W.F.~Loyd \cite{Lloyd_1833}.
    Alternatively, we could have chosen
    the terminology \emph{cooperation},
    which is often used in the literature
    of mean field optimisation,
    to refer to the maximization of the
    cumulated profit over all agents
    (which is typically the case in a monopoly situation)
    Here, the wording \emph{coordination}
    emphasises the fact that the benefit of
    every agent is increased at an individual
    scale, and is in line with our qualitative objectives. We refer to Remark \ref{Re:Coop} for more details.

We want to prove that, when $(\lambda,c)$
and the lagrangian $L$ are chosen in such a way that
a reversed monotonous MFG travelling wave $(\cV,\Th,\M)$
(provided by Theorem~\ref{Th:small_lambda}),
it is sometimes possible to find a coordinated strategy
$\alpha_{\mathrm{co}}\in L^\infty(\R)$ such that:
\begin{enumerate}
\item Each fisherman actually obtains a higher harvest than in the competitive case,
\item The fishes' population does not go extinct, and is even invading.
\end{enumerate}
In our next result we prove that, provided $L$ is chosen carefully, this is possible. The general question (\emph{i.e.} is it always possible, whenever a reversed MFG travelling exists, to find a common strategy that outperforms the constant strategy $\overline\alpha\equiv c$?) remains open at this stage.
We make the conjecture that the answer is yes for any $L$ satisfying
\eqref{Hyp:L} and $\lambda$ small enough.

\begin{theorem}\label{Th:Cooperation}
There exist a Lagrangian $L$, a discount factor $\lambda$ and $c>0$ such that:
\begin{enumerate}
\item There exists a reversed MFG travelling wave $(c,\lambda,\M,\Th,\cV)$,
\item  There exists an explicit strategy $\alpha_{\mathrm{co}}\in L^\infty(\R)$ that satisfies:
            \begin{equation*}
                \cV(x_0)
                <
                J(x_0,\alpha_{\rm co},\th_{\alpha_{\rm co}})
                \;\text{ where }\;
                \begin{cases}
                    \partial_t\th_{\alpha_{\rm co}}
                    -\partial_{xx}^2\th_{\alpha_{\rm co}}
                    =
                    f(\th_{\alpha_{\rm co}})
                    -m_{\alpha_{\rm co}}\th_{\alpha_{\rm co}},
                    \;
                    \th_{\alpha_{\rm co}}(0)
                    =
                    \th_0,
                    \\
                    \partial_t m_{\alpha_{\rm co}}
                    +\partial_x(\alpha_{\rm co} m_{\alpha_{\rm co}})
                    =
                    0,
                    \;
                    m_{\alpha_{\rm co}}(0)
                    =
                    m_0.
                \end{cases}
            \end{equation*}
\end{enumerate}
\end{theorem}
\begin{remark}\label{Re:Coop}
    While \emph{mean field game}
    commonly refers to a situation where the
    agents are in competition with each others,
    its cooperative counterpart is known
    as \emph{mean field control} (MFC), see \cite{Bensoussan_2013}.
    It consists in maximising the following
    cumulative benefit
    \begin{equation*}
        \sup_{\alpha\in L^{\infty}(\R)}
        \int_{\mathbb{R}}
        J(x_0,\alpha,\th_{\alpha})
        m_0(x_0)dx_0
        \;\text{ where }\;
        \begin{cases}
            \partial_t m_{\alpha}
            +\partial_x(\alpha m_{\alpha})
            =
            0,
            \\
            \partial_t\th_{\alpha}
            -\partial_{xx}^2\th_{\alpha}
            =
            f(\th_{\alpha})
            -m_{\alpha}\th_{\alpha}.
        \end{cases}
    \end{equation*}
    The main difference between this setting and
    the MFG one \eqref{Eq:OCPlayer}
    is that $\th$ is frozen in~\eqref{Eq:OCPlayer}.
In the cooperation context, the MFC system writes    \begin{equation}\label{Eq:MFC}
        \begin{cases}
        \lambda V+\partial_t V-H(\partial_xV)
        =
        \theta(1-\eta),
        &V(\infty,\cdot)\equiv 0,
        \\
        \partial_tm
        +\partial_x\left(H'(\partial_x V) m\right)
        =
        0,
        &m(0,\cdot)
        =
        m_0,
        \\
        \partial_t \theta
        - \Delta\theta
        -f(\theta)
        +m\theta
        =
        0,
        &\theta(0,\cdot)=\theta_0,
        \\
        -\partial_t \eta
        - \Delta\eta
        -\eta(f'(\theta)-m)
        =
        m,
        &\eta(T,\cdot)=0.
\end{cases}
\end{equation}
    Under similar assumptions as in Theorem \ref{Th:Cooperation},
    we expect that the solution to the MFC system
    leads the fishes' population to be invading as well,
    but a rigorous proof is out of reach in the present work. 
    For more details on the derivation of System \eqref{Eq:MFC}
    and some theoretical results,
    in the case of a monostable nonlinearity,
    we refer to our other work \cite{KMFRB2023}, which also contains a discussion of possible governmental regulations.

\end{remark}

\subsection{Extensions: non-monotonous reversed MFG travelling waves}\label{Se:Extensions}
We study in this paragraph non-monotonous reversed MFG travelling waves, and we highlight the fact that this ``non-monotonicity" can also mean ``periodic"; we split the results accordingly;

\subsubsection{Non-monotonous solutions with an interval as the support of the fishermen density}
\label{subsec:MR_nonmono}
We begin this section with the analog of Proposition~\ref{Prop:exist_th_TW}; namely, we establish the existence of reversed travelling waves with an arbitrary number of local maxima:
\begin{proposition}
    \label{Prop:exist_th_multiple_bumps}
For any positive integer $k$, there exists a map $c_k:(0;\infty)\ni \mapsto c_k(\lambda)$ such that    \begin{enumerate}[i)]
        \item
    for any
    $c\in(0,c_k(\lambda))$,
    there exists a unique
    couple $(\Th,\M)$
    such that
    \begin{itemize}
        \item
            $\M\in L^1(\R;[0,\infty))$
            and $\supp(\M)=[0,s_1]$
            for some $s_1>0$,
        \item
            $\Th$ is piecewise $\mathscr C^2$,
            admits exactly
            $k$ bumps,
            and goes
            from $0=\displaystyle{\lim_{s\to-\infty}}\Th(s)$
            to $1=\displaystyle{\lim_{s\to\infty}}\Th(s)$,
        \item
            $\Th$ satisfies
            $-c\Th'-\mu\Th''=f(\Th) - \M\Th$
            on $\R$,
        \item
            $\Th'\equiv \lambda L'(c)$
            on $\supp(\M)$,
        \end{itemize}
    \item
        $c_k$ is continuous decreasing, with
        $\lim_{\lambda\to 0^+}c_k(\lambda)=\max(c_{\max},2\sqrt{f'(\eta)})$ and $\lim_{\lambda \to \infty}c_k(\lambda)=0$.
    \end{enumerate}
\end{proposition} 

The main result of this section reads as follows:
\begin{theorem}\label{Th:exist_nonmono}
    Let $k\in \N^*$.    There exists a map 
    $\lambda \mapsto \e_k(\lambda)>0$
    such that,
    for any $\lambda>0$ and any $c\in(0,\e_k(\lambda))$
    there exists a unique
    reversed MFG travelling wave
     that has exactly $k$ local maxima and such that $\mathrm{supp}(\M)$ is an interval.
\end{theorem}
We refer to Fig.~\ref{fig:sol_nonmono}
for a graphical representation
of a non-monotonous MFG travelling wave.

As in the case of monotonous reversed travelling waves, we might look into the sharpness of $c_k$ as an upper bound for the existence of a reversed MFG travelling wave. Similar to Theorem~\ref{Th:Quadratique} we obtain 
\begin{theorem}\label{Th:Quadratique2}

    Assume that $L$ is $\mathscr C^2$
    and $L''\geq \underline D$
    for some positive constant $\underline D$.
    There exists $\lambda_k(\underline D)>0$
    such that, for any $(c,\lambda)$ satisfying
    \begin{equation*}
        \lambda\geq\lambda_k(\underline D)
        \qquad
        \text{ and }
        \qquad
        0<c\leq c_k(\lambda),
    \end{equation*}
    there exists a unique non-monotonous 
    reversed MFG travelling wave with exactly $k$ bumps, 
    with velocity $c$, a discount factor $\lambda$ and such that $\mathrm{supp}(\M)$ is an interval.

\end{theorem}

\begin{figure}[h!]
\begin{center}
\includegraphics[scale=0.12]{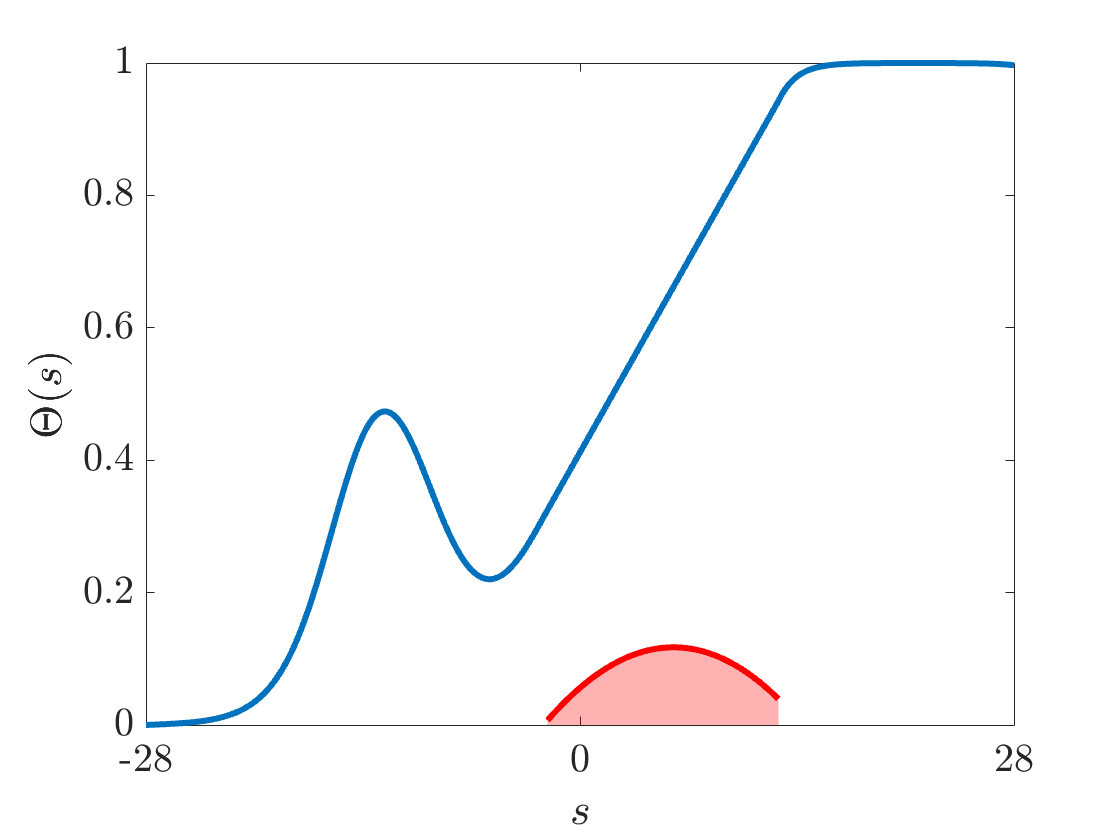}
\includegraphics[scale=0.12]{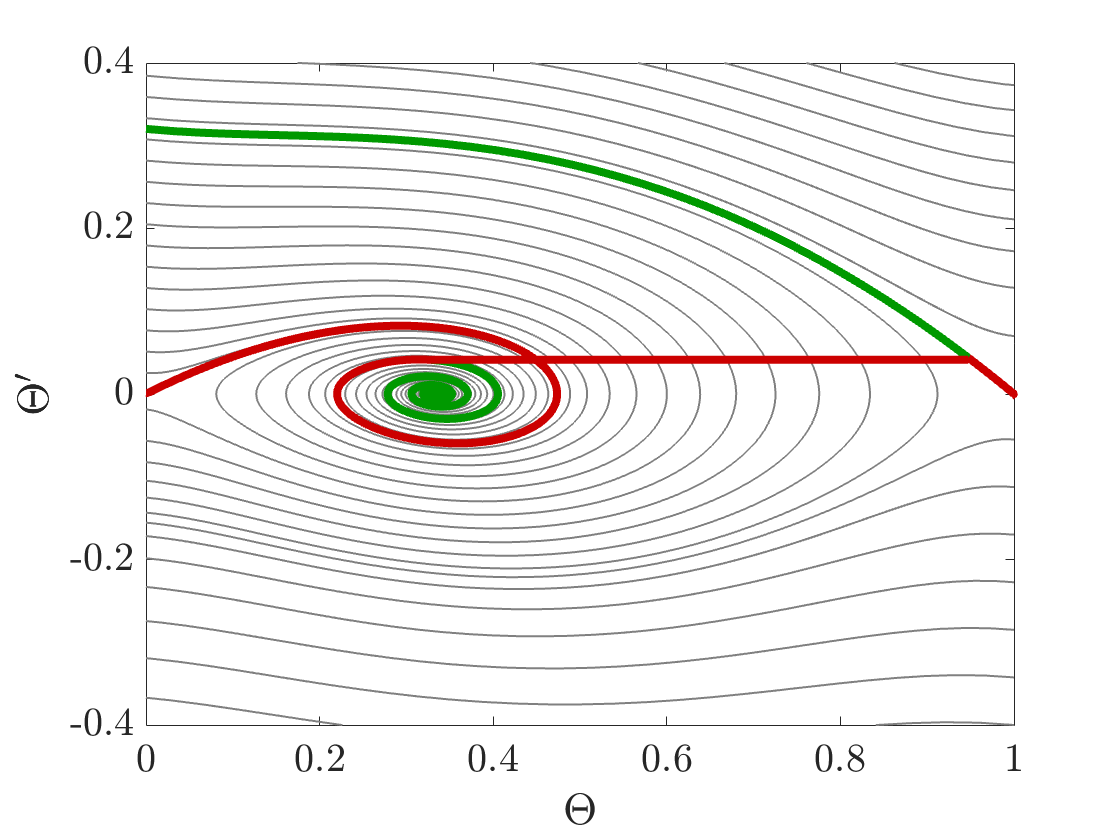}
\includegraphics[scale=0.12]{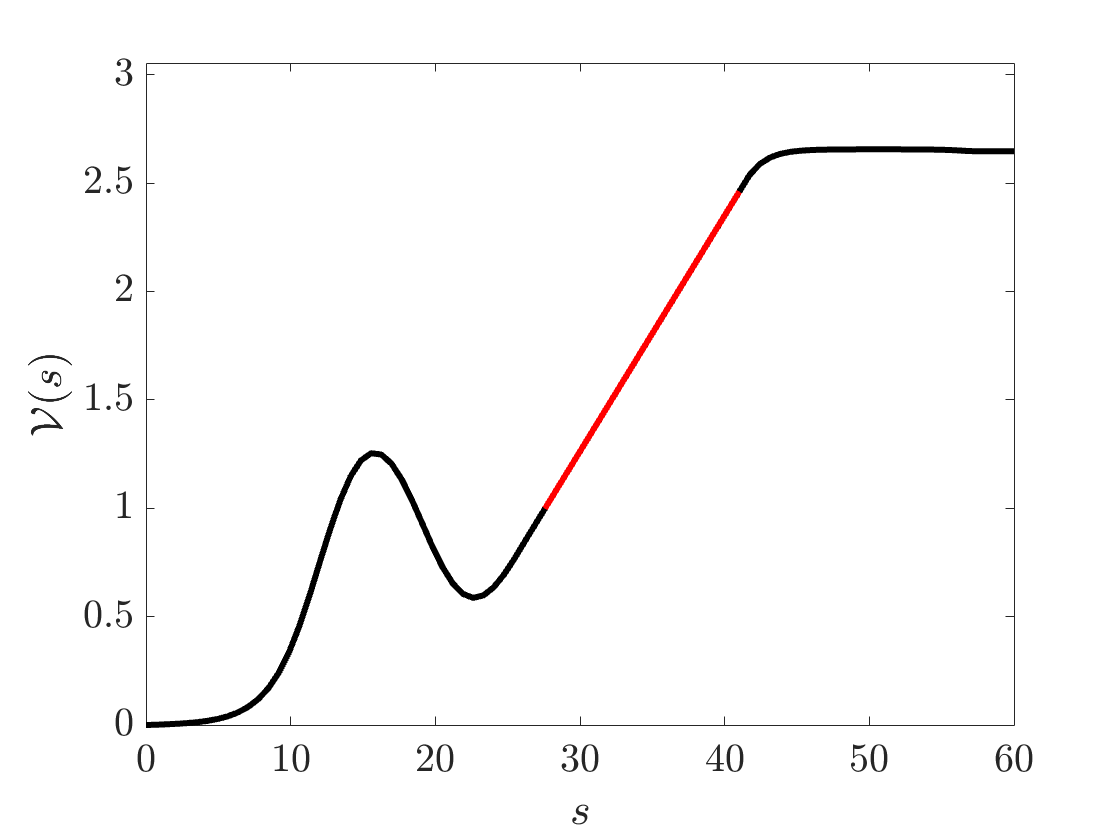}
\end{center}
\caption{
    \label{fig:sol_nonmono}
    Representation of a non-monotonous
reversed MFG travelling wave:
$\Th$ in blue and $\M$ in red
(left); construction  in the phase portrait (center); representation of the value function (right), the red segment where the support of $\mathcal{M}$ is located,  $\lambda=0.39$.}
\end{figure}

\subsubsection{Periodic solutions}
\label{subsec:MR_periodic}
In this section,
we are interested in periodic solutions
so that we drop the assumption that $\Th$
has limits 0 and 1 at $s=\pm\infty$.

As will be clear when building non-monotonous solutions, the basic idea behind the construction of periodic reversed MFG travelling waves is to ``copy and paste" parts of a non-monotonic solutions, see Fig.~\ref{fig:sol_periodic}
for a graphical representation.
This leads to the following
theorems, the proof of which is a straightforward adaptation of 
Theorem~\ref{Th:exist_nonmono}.
\begin{theorem}\label{Th:exist_periodic}
    
    For any $\lambda>0$, for any $c\in(0,\e_1(\lambda))$
    (where $\e_1$ is defined in Theorem 
   ~\ref{Th:exist_nonmono})
    there exists a
    periodic MFG travelling wave (in the sense that $\Th$ and $\M$ are periodic in space).
\end{theorem}

\begin{figure}[h!]
\begin{center}
\includegraphics[scale=0.12]{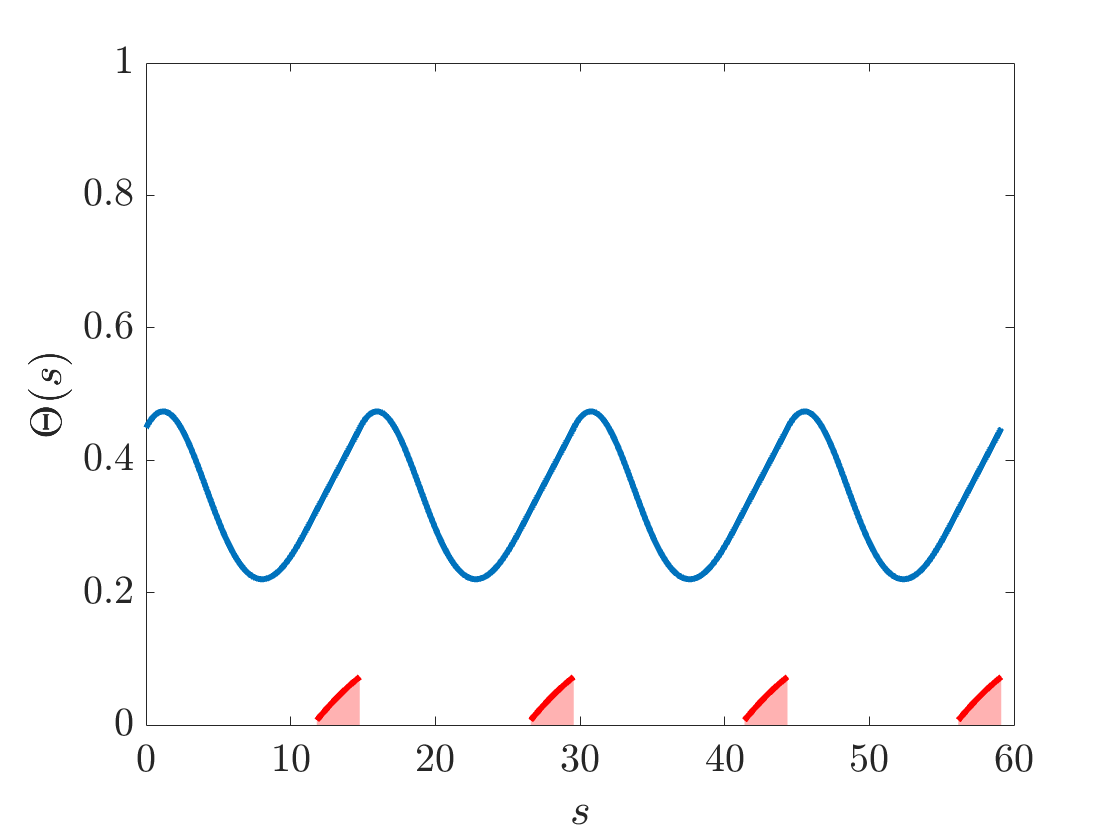}
\includegraphics[scale=0.12]{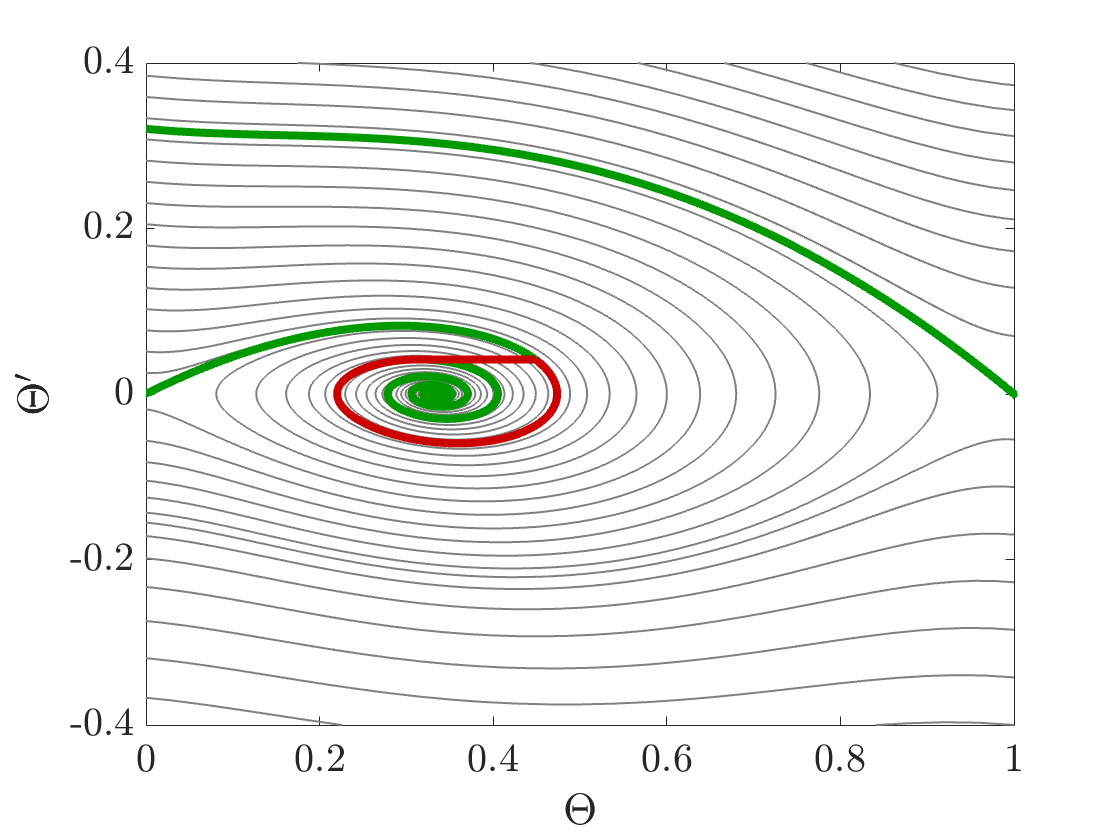}
\includegraphics[scale=0.12]{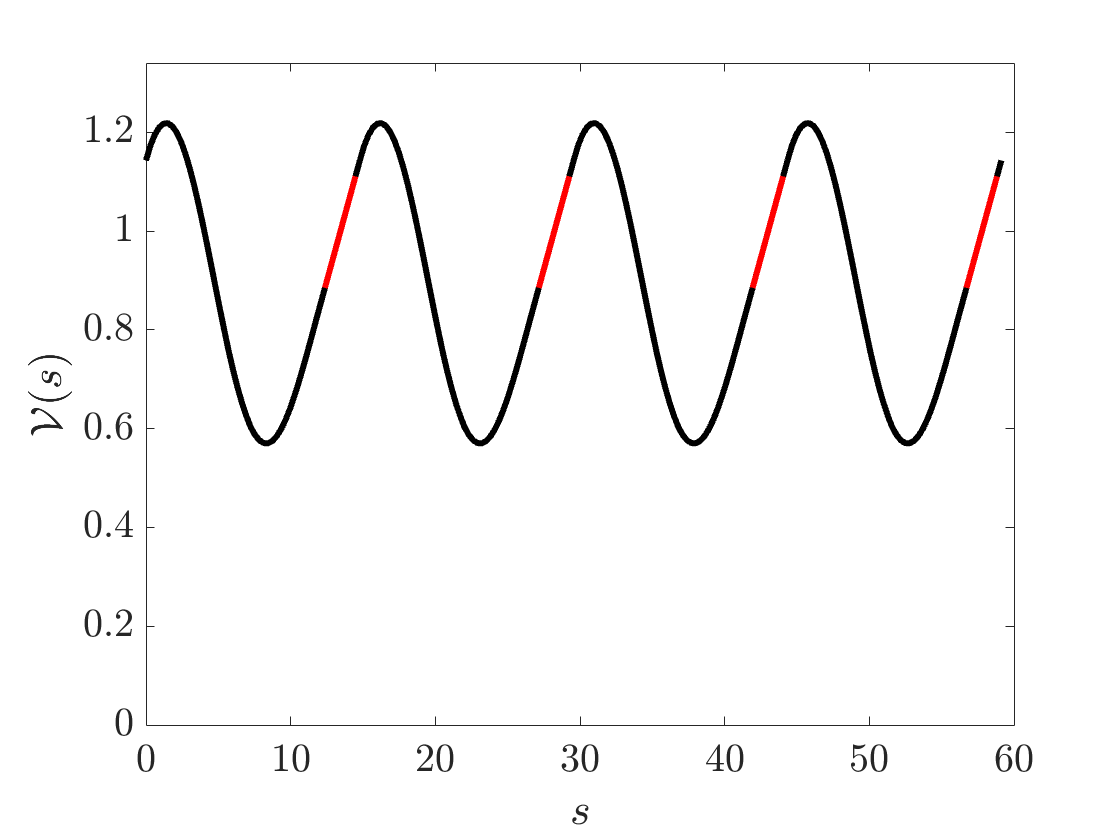}
\end{center}
\caption{
    \label{fig:sol_periodic}
    Representation of a periodic
reversed MFG travelling wave:
$\Th$ in blue and $\M$ in red
(left); Construction in the phase portrait; value function  (right) the red segment where the support of $\mathcal{M}$ is located, $\lambda=0.39$.}
\end{figure}

\subsection{Bibliographical references}
\label{subsec:related_works}
Our article both fits in three main mathematical streams. The first one is, naturally, the study of spatial ecology through the properties of travelling waves. It would be pointless to try and give an exhaustive list of references dealing with this topic, but let us underline that, in this paper, we focus on the basic theory, as exposed in the classic monograph of Fife \cite{Fife1979}; we refer for instance to \cite{Garnier_2012} for a recent detailed study of the dynamics of travelling waves. 
The second one is Mean-Field Games theory \cite{MR2346927,MR2352434,MR2269875,MR2271747,MR2295621}, which has in the past two decades become one of  the main fields of applied partial differential equations. In our paper, the main difference with the usual literature devoted to the study of MFG is of course the presence of an additional equation, corresponding to the fishes' population. As we alluded to earlier, our paper is not the first to investigate possible links between MFG systems and travelling waves: there is a growing literature devoted to the study of travelling waves in Boltzmann type MFG models for knowledge growth arising in economics sciences \cite{Burger_2017,Papanicolaou_2021,Porretta_2022,Qin_2019}. Here again, our model strongly differs, as do its applications. As far as we are aware, the only other instance of coupling of reaction-diffusion equations and Mean Field Games is \cite{e4bc8ffa37dd44639b7657ba22baab64}; however, from what we could infer from the set of slides \cite{e4bc8ffa37dd44639b7657ba22baab64} our point of view is quite different.

The third field is that of optimal control problems in population dynamics. This field is rapidly growing, and several recent contributions have, specifically, investigated some qualitative properties of optimal harvest problems.
Typically, in \cite{Bressan_2013,Coclite_2017}, some optimal control problems for the harvesting of populations governed by elliptic equations are analysed in details, using a measure solution paradigm. More recently, some attention has been devoted to the interplay between the optimal control of parabolic equations in unbounded domains and the underlying travelling waves dynamics. In particular, Bressan, Chiri and Salehi \cite{Bressan} have provided an in-depth description and explanation of the possibility to reverse travelling waves for monostable or bistable equations using phase-portrait methods reminiscent of the ones we use here. In the recent work \cite{almeida:hal-03811940}, Almeida, L\'eculier, Nadin and Privat also tackled the problem of knowing whether or not it was possible, using a control acting as a harvesting term, to block the invasion of a pest. These articles, while featuring a notion of optimality (in particular, minimising a certain norm of the control), do not however provide conclusive answers to the question under scrutiny here, since they do not take into account the optimality from the harvesters' perspectives. As far as we are aware, our previous paper \cite{Mazari_2022} was the first systematic analysis of such optimal distributed harvest problems from the point of view of Nash equilibria for a finite number of players, albeit in the case of elliptic, monostable equations set in bounded domains. While we also explored some facet of the tragedy of the commons in \cite{Mazari_2022}, both the methods and the qualitative flavour of results presented here greatly differ from that of \cite{Mazari_2022}.

\section{Construction of possible $(\lambda,c,\Th,\M)$}
\label{sec:construct_Th}
This section is devoted to the construction of possible candidates $(\lambda,c,\Th,\M)$ that could provide reversed MFG travelling waves. In particular, we prove Proposition~\ref{Prop:exist_th_TW} as well as its extension to the non-monotonous case, Proposition~\ref{Prop:exist_th_multiple_bumps}. Recall that for a fixed $c>0$ our goal is to find couple $(\Th,\M)$, with $\M\in L^1(\R)$ compactly supported, that solves the dynamical system
\begin{equation}\label{Eq:ODE}
\begin{cases}
\frac{d}{ds} \begin{pmatrix} \Th\\ \Th' \end{pmatrix} = \begin{pmatrix}\Th'\\ -f(\Th)-c\Th'+\M\Th \end{pmatrix}\\
\begin{pmatrix}
\Th(-\infty)=0\\
\Th'(-\infty)=0
\end{pmatrix},\quad \begin{pmatrix}
\Th(+\infty)=1\\
\Th'(+\infty)=0
\end{pmatrix},
\end{cases}
\end{equation} where $s=x-ct$ is the similarity variable.
 
 The main tool we use is a phase plane analysis of the underlying dynamical systems when $\M\equiv 0$; the control $\M$ will be used as a control so as to ``glue" together certain parts of this phase portrait.
 Standard tools as the phase plane analysis of bistable equations
 may be used  (see \emph{e.g.} \cite{FifeMcLeod}),
 we give proofs in order to be as self contained as possible.

\subsection{Phase portrait analysis when $\M\equiv 0$}
\paragraph{Goal of the section}
We consider a fixed $c>0$; our goal is to prove the existence of two trajectories $\Gamma_{0,c}\,, \Gamma_{1,c}$ that satisfy the following differential equations:
\begin{equation}\label{Eq:MB}
\begin{cases}
-\Gamma_{0,c}''-c\Gamma_{0,c}'=f(\Gamma_{0,c})&\text{ in }\R\,, 
\\ 0<\Gamma_{0,c}<1\,,
\\ \Gamma_{0,c}(-\infty)=0,\Gamma_{0,c}'(-\infty)=0\end{cases}\end{equation}
and 
\begin{equation}\label{Eq:ME}
\begin{cases}
-\Gamma_{1,c}''-c\Gamma_{1,c}'=f(\Gamma_{1,c})&\text{ in }\R\,, 
\\0< \Gamma_{1,c}<1\,,
\\ \Gamma_{1,c}(+\infty)=1\,, \Gamma_{1,c}'(+\infty)=0.\end{cases}\end{equation} 

 Once the existence of these two solutions is established, the strategy to build control $\M$ is transparent: first, starting from $s=-\infty$, we simply follow the trajectory $\Gamma_{0,c}$ (\emph{i.e} we set $\M=0$); second,  we choose a certain $s_0$ (which, up to translation, is assumed to be equal to 0) and we build $\mathcal M$ so as to connect $\Gamma_{0,c}$ to $\Gamma_{1,c}$.

To prove the existence of $\Gamma_{0,c}\,, \Gamma_{1,c}$, it is convenient to introduce the corresponding differential system as a first-order ODE as follows:
\begin{equation}\label{Eq:ODE2}
\frac{d}{ds}\begin{pmatrix}U\\U'\end{pmatrix}=\begin{pmatrix}U'\\ f(U)-cU'\end{pmatrix}.
\end{equation} This system has exactly three equilibria: 
\begin{equation}\label{Eq:Equilibria} \begin{pmatrix} 0\\ 0\end{pmatrix}\,, \begin{pmatrix}\eta\\0\end{pmatrix}\,, \begin{pmatrix} 1\\0\end{pmatrix}.\end{equation}
 From this point of view, the functions $\Gamma_{0,c}\,, \Gamma_{1,c}$ will be constructed as (un)stable manifolds associated to these equilibria. To be more precise, $\Gamma_{0,c}$ will correspond to the \emph{unstable manifold} associated with the equilibrium 
$(0,0)^{\top}$
while $\Gamma_{1,c}$ will correspond
to the \emph{stable manifold} associated with the equilibrium 
$(1,0)^{\top}$.

\paragraph{Equilibria and (un)stable manifolds of \eqref{Eq:ODE2}}
Before giving our main results about the (un)stability of the equilibria, we introduce the energy functional
\begin{equation}\label{Eq:Energy}E(u,p):=\frac{p^2}2+F(u)\text{ where }F:x\mapsto \int_0^x f.\end{equation} Setting, for any trajectory $(U(t),U'(t))^{\top}$ of \eqref{Eq:ODE2}, 
\[ e(t):=E(U(t),U'(t))\] we obtain 
\begin{equation}\label{Eq:EnergyDissipates} e'(t)=-c (U')^2(t)\leq 0.\end{equation} The last inequality comes from the condition $c>0$. This leads to considering the invariant region 
\begin{equation}\label{Eq:InvariantSet} X:=\{(u,p)\,, E(u,p)\leq 0\}, \end{equation} whose boundary in the phase plane is given by 
\[ \partial X=\{(u,p)\,,0\leq u\leq 1\,, F(u)\leq 0\,, p=\pm \sqrt{-2F(u)}\}.\]

The following lemma describes the stability of the equilibria given in \eqref{Eq:Equilibria}. 
 Each of this equilibria is of the form
 $(u^*,0)^{\top}$,
 so that the linearisation of \eqref{Eq:ODE2} at
 $(u^*,0)^{\top}$
 is the dynamical system
\begin{equation}
\frac{d}{ds}\begin{pmatrix}V\\cV'\end{pmatrix}=A_{u^*}\begin{pmatrix}V\\cV'\end{pmatrix}
\end{equation} where the matrix $A_{u^*}$ is defined as 
\[A_{u^*}=\begin{pmatrix}0&1\\ -f'(u^*)&-c\end{pmatrix}.
\] It is readily checked that, for $u^*=0$ or $1$, since $f'(0)\,, f'(1)<0$, the matrix $A_{u^*}$ has two real eigenvalues $(\lambda_{+,u^*}\,, \lambda_{-,u^*})$, associated with two eigenvectors $(v_{+,u^*},v_{-,u^*})$, whose expressions are given explicitly by 
\begin{equation}\label{Eq:Eigen}\lambda_{\pm,u^*}=\frac{-c\pm\sqrt{c^2-4f'(u^*)}}{2}\,, v_{\pm,u^*}=\begin{pmatrix}1\\ \lambda_{\pm,u^*}\end{pmatrix}.\end{equation}
 Since $\lambda_{+,u^*}>0, \lambda_{-,u^*}<0$ when $u^*=0,1$ we deduce that $(0,0)$ and $(1,0)$ are saddle points, and that there is one unstable manifold and one stable manifold associated with them in the following sense:

 \begin{lemma}\label{Le:Stab1} For any $c>0$, the following hold:
\begin{enumerate}
\item There exists a curve $S_0:=(\Gamma_{0,c},\Gamma_{0,c}')^{\top}$,
    dubbed the ``unstable manifold'',
    that satisfies \eqref{Eq:MB}.

Furthermore, it
is included in the invariant region: $S_0\subset X$
where $X$ is defined in \eqref{Eq:InvariantSet}. Finally, we have
\[ \lim_{s\to +\infty}\begin{pmatrix}\Gamma_{0,c}(s)\\\Gamma_{0,c}'(s)\end{pmatrix}=\begin{pmatrix}\eta\\0\end{pmatrix}.\]
\item  There exists a curve
    $S_1:=\begin{pmatrix}\Gamma_{1,c},\Gamma'_{1,c}\end{pmatrix}^{\top}$,
    dubbed the ``stable manifold'',
    that satisfies \eqref{Eq:ME}.
Furthermore,
it is included in $X^c$, and  it satisfies $\Gamma_{1,c}'>0$ in $\R$. Finally, there exists $s_\eta>0$ such that $\Gamma_{1,c}(s_\eta)=\eta$.

\end{enumerate}
\end{lemma}

\begin{proof}[Proof of Lemma~\ref{Le:Stab1}]
\begin{enumerate}
\item To check that $S_0\subset X$, it suffices to observe that  $S_0\cap\{u\leq \e\}\subset X$ for $\e>0$ small enough. However, we already observed that $\partial X=\{(u,p)\,, p=\pm \sqrt{-2F(u)}\}$. Consider the function $\Phi=u\mapsto \sqrt{-2F(u)}$ in the phase plane. Then, for any $u>0$ such that $F(u)<0$, $\Phi'(u)=-\frac{F'(u)}{\sqrt{-2F(u)}}\underset{u\to 0}\to \sqrt{-f'(0)}.$ From the explicit expression of $v_{-,0}$, it is sufficient to check that 
\[-c-\frac{\sqrt{c^2-4f'(0)}}2\leq \sqrt{-f'(0)},\]
which is obviously true for any $c>0$.
To check that the trajectory converges to
$(\eta,0)^{\top}$,
simply observes that $E$ is a strict Lyapunov functional for \eqref{Eq:ODE2}, and that it has a unique global (strict) minimum at
$(\eta,0)^{\top}$.
The proof is concluded.
\item As $E(1,0)=F(1)=\int_0^1 f>0$, and from the monotonousity of the energy along trajectories~\eqref{Eq:EnergyDissipates}, we know that $E(\Gamma_{1,c},\Gamma_{1,c}')>0$. We also note that the same property, combined with the fact that $F$ reaches its maximum at $1$, implies that $\Gamma_{1,c}'(s)>0$ for any $s\in \R$.
Furthermore, there exists $s_\eta$ such that $\Gamma_{1,c}(s_\eta)=\eta$. Indeed, if this were not the case, we would have 
\[\Gamma_{1,c}>\eta\text{ in }\R.\] Since, as observed, $\Gamma_{1,c}'>0$ in $\R$, this would imply that 
\[ \Gamma_{1,c}''=-c\Gamma_{1,c}'-f(\Gamma_{1,c})<0.\] Thus, $\Gamma_{1,c}$ would be concave, strictly increasing at $+\infty$ and bounded from below by $\eta$ at $-\infty$. This is absurd, and we thus conclude that such an $s_\eta$ exists.
\end{enumerate}
\end{proof}

 The dynamical behaviour of the third equilibria
$(\eta,0)^{\top}$.
 is more involved and is described in the next lemma (which follows from the study of \eqref{Eq:Eigen}):
\begin{lemma}\label{Le:Stab2}
\begin{enumerate}
\item If $0<c<2\sqrt{f'(\eta)},$ the equilibrium
    $(\eta,0)^{\top}$
    is a spiral sink: the linearised system has two conjugate complex eigenvalues with negative real parts.
\item If $2\sqrt{f'(\eta)}<c$,
    the equilibrium
    $(\eta,0)^{\top}$
    is locally stable: the linearised system has two negative real eigenvalues.
\end{enumerate}
\end{lemma}

\subsection{Proof of Proposition~\ref{Prop:exist_th_TW}: construction of a monotonous reversed travelling wave}
We now move on to the proof of Proposition~\ref{Prop:exist_th_TW}.
We introduce the following definition:
\begin{definition}\label{De:ReversedTW}
For any $c>0$, we say that $(\Th,\M)$ is a monotonous reversed travelling wave with speed $c$ if $\M\in L^1(\R)\,, \M\geq 0$, and $(\Th,\M)$  solves
\begin{equation}\begin{cases}
\frac{d}{ds}\begin{pmatrix}
\Th\\\Th'
\end{pmatrix}=\begin{pmatrix}\Th'\\ -f(\Th)-\M\Th+c\Th'\end{pmatrix}\,,
\\ \begin{pmatrix}
\Th(-\infty)\\\Th'(-\infty)
\end{pmatrix}=\begin{pmatrix}
0\\0
\end{pmatrix}\,, \begin{pmatrix}
\Th(+\infty)\\\Th'(+\infty)
\end{pmatrix}=\begin{pmatrix}
1\\0
\end{pmatrix}
\end{cases}
\end{equation} with $\Th'>0$ in $\R$.
\end{definition}
\paragraph{Preliminary discussion}
There are several ways to construct such reversed travelling waves, all of them following a simple strategy: follow the unstable manifold $S_0$ associated with the equilibrium $(0,0)$, and then build a compactly supported $\mathcal M$ taking you from $S_0$ to the stable manifold $S_1$ associated with $(1,0)$. Once $S_1$ is reached, stop acting and simply follow $(\Gamma_{1,c},\Gamma_{1,c}')$.

However, as pointed out in Lemma~\ref{Le:MonotComp}, if we want a reversed travelling wave that is a potential candidate to be a solution of the MFG problem \eqref{Eq:MFG}, we can not choose just any reversed travelling wave. We must indeed ensure that $\M$ has a compact support and that
\[\cV'=L'(c)\] in the support of $m$, which in turns implies that $\Th'$ must be constant in the support of $\M$. More precisely, we must have the relationship
\[ L'(c)=\lambda \Th'\text{ in }\mathrm{supp}(\M).\]
This imposes a first constraint on the range of $(c,\lambda)$.

\paragraph{Construction of $c_0$}
We now define the map $c_0$ used in Proposition~\ref{Prop:exist_th_TW}. For any fixed $c>0$, we define the quantities
\begin{equation}\label{Eq:Gamma0Max}
\Gamma_{0,\max}(c):=\Vert \Gamma_{0,c}\Vert_{L^\infty(\R)}\end{equation} and 
\begin{equation}\label{Eq:Gamma_0PrimeMax}
\Gamma'_{0,\max}(c):=\max_{\R} \Gamma_{0,c}'.\end{equation}
In fact we only need $\Gamma'_{0,\max}$ to define $c_0$; however, as $\Gamma_{0,\max}$ is used in a later construction, it is convenient, for further reference, to define it here.

The following lemma is key:
\begin{lemma}\label{Le:Galere1}
The map $\Gamma'_{0,\max}$ is decreasing on $(0;+\infty)$. Furthermore, 
\[ \lim_{c \to 0^+}\Gamma_{0,\max}'(c)>0\,, \lim_{c \to \infty}\Gamma_{0,\max}'(c)=0.\]
\end{lemma}

\begin{proof}[Proof of Lemma~\ref{Le:Galere1}]
We begin with the strict monotonicity of $\Gamma_{0,\max}'$. Let $c_1\,, c_2>0$ with $c_1<c_2$. For $i=1,2$, we define 
\[  s_i:=\sup\left\{s\in \R:\, \forall s'\leq s\,, \Gamma_{0,c_i}(s')\leq \eta \right\}.\] Observe that, while $s_i>-\infty$, we might have $s_i=\infty$ (this is the case if, for instance, $\Gamma_{0,c_i}$ is monotonous). At any rate, first observe that 
\begin{equation}\label{Eq:Si1} \Gamma_{0,c_i}'>0\text{ in }(-\infty,s_i).\end{equation}
To establish \eqref{Eq:Si1}  observe that the weaker condition 
\begin{equation}\label{Eq:Si2} \Gamma_{0,c_i}'\geq 0\text{ in }(-\infty,s_i).\end{equation}
holds. Indeed, should \eqref{Eq:Si2} not hold, pick $s^*<s_1$ such that 
\[ \Gamma_{0,c_i}'(s^*)=\min_{s \leq s_i} \Gamma_{0,c_i}'(s)<0.\] From \eqref{Eq:MB} we deduce that 
\[ 0<-c \Gamma_{0,c_i}'(s^*)=f\left(\Gamma_{0,c_i}(s^*)\right).\] Consequently, $\Gamma_{0,c_i}(s^*)> \eta$, in contradiction with the definition of $s_i$. Thus, \eqref{Eq:Si2} is valid. To prove \eqref{Eq:Si1}, argue once again by contradiction and assume there exists $s^*<s_i$ such that $\Gamma_{0,c_i}'(s^*)=0$.
System \eqref{Eq:MB} implies that either $\Gamma_{0,c_i}(s^*)=0$ or $\Gamma_{0,c_i}(s^*)=\eta$; the first possibility is ruled out as $\Gamma_{0,c_i}>0$ in $(-\infty;+\infty)$, while the second contradicts the definition of $s_i$.

The second observation is that 
\begin{equation}\label{Eq:Si3}\max_\R\Gamma_{0,c_i}'=\max_{(-\infty;s_i)}\Gamma_{0,c_i}'.\end{equation} Indeed, define \[ s_{i,\mathrm{max}}:=\mathrm{arg min}\left\{ s\in \R: \Gamma_{0,c_i}'(s_{\mathrm{max}})=\max_\R \Gamma_{0,c_i}'.\right\}\] By \eqref{Eq:MB}  we have 
\[0>-c_i\Gamma_{0,c_i}'(s_{i,\mathrm{max}})=f\left(\Gamma_{0,c_i}(s_{i,\mathrm{max}})\right).\] Thus, we obtain $\Gamma_{0,c_i}(s_{i,\mathrm{max}})<\eta$.
We claim that 
\begin{equation}\label{Eq:Si4}\Gamma_{0,c_i}'\geq 0\text{ in }(-\infty;s_{i,\mathrm{max}}].\end{equation}  Indeed, if this is not the case, we let $s^*<s_{i,\mathrm{max}}$ such that $\Gamma_{0,c_i}'(s^*)=\min_{(-\infty;s_{i,\mathrm{max}})} \Gamma_{0,c_i}'<0$. Equation~\eqref{Eq:MB} in turn implies that 
\[\Gamma_{0,c_i}>\eta.\] Since $\Gamma_{0,c_i}(s_{i,\mathrm{max}})<\eta$, there exists $s^{**}<s^*$ such that $\Gamma_{0,c_i}(s^{**})=\Gamma_{0,c_i}(s_{i,\mathrm{max}})$ and $\Gamma_{0,c_i}'(s^{**})>0$. However, in view of \eqref{Eq:Energy}, define 
\[ e_i:s\mapsto \frac{\Gamma_{0,c_i}'(s)^2}2+F(\Gamma_{0,c_i}(s)).\] It is obvious that 
\[ e_i'(s)=-c\left(\Gamma_{0,c_i}'(s)\right)^2,\] whence 
\[ 0>e_i(s_{i,\mathrm{max}})-e_i(s^{**})=\frac{(\Gamma_{0,c_i}'(s_{i,\mathrm{max}}))^2-(\Gamma_{0,c_i}'(s^{**}))^2}2,\] in contradiction with the definition of $s_{\mathrm{max}}$.

From \eqref{Eq:Si1}, $\Gamma_{0,c_i}$ is a diffeomorphism from $(-\infty;s_i)$ onto $(0;\eta)$. We now define $\tilde\gamma_i'$ as 
\begin{equation}\label{Eq:DefGam}
\tilde\gamma_i':(0;\eta)\ni \xi \mapsto \Gamma_{0,c_i}'\left(\Gamma_{0,c_i}^{-1}(\xi)\right).
\end{equation}
Given \eqref{Eq:Si3}, it suffices to prove that \begin{equation}\label{Eq:Ke}\tilde\gamma_2'<\tilde\gamma_1'\end{equation} to conclude that $\Gamma_{0,\max}'$ is decreasing. By \eqref{Eq:Eigen}, 
\[\tilde\gamma_i'(0)=\frac{-c_i+\sqrt{c_i^2-4f'(0)}}2.\] However, it is readily checked that $c\mapsto\frac{-c+\sqrt{c^2-4f'(0)}}2$ is decreasing in $c$. Thus
\begin{equation}\label{Eq:Si4}
\tilde\gamma_2'(0)<\tilde\gamma_1'(0).\end{equation} Now, to prove \eqref{Eq:Ke}, argue by contradiction. Using \eqref{Eq:Si4}, this implies the existence of $\xi^*\in (0;\eta)$ such that 
\[ \tilde\gamma_1'(\xi^*)=\tilde\gamma_2'(\xi^*)\text{ with }\tilde\gamma_2'<\tilde\gamma_1'\text{ in }[0;\xi^*),\] which in turn leads to the existence of $\xi^{**}\in (0;\xi^*)$ such that 
\[ \frac{d \tilde\gamma_2'}{d\xi}(\xi^{**})>\frac{d\tilde\gamma_1'}{d\xi}(\xi^{**})\,, \tilde \gamma_2'(\xi^{**})<\tilde\gamma_1'(\xi^{**}).\] Going back to the definition of $\tilde\gamma_i$ (Eq. \eqref{Eq:DefGam}) and using \eqref{Eq:MB} this yields the comparison
\[\frac{-c_2\tilde\gamma_2'(\xi^{**})-f(\xi^{**})}{\tilde\gamma_2'(\xi^{**})}=\frac{d \tilde\gamma_2'}{d\xi}(\xi^{**})>\frac{d\tilde\gamma_1'}{d\xi}(\xi^{**})=\frac{-c_1\tilde\gamma_1'(\xi^{**})-f(\xi^{**})}{\tilde\gamma_1'(\xi^{**})}
\] or, equivalently, 
\[ -c_2 \tilde\gamma_1'(\xi^{**})>-c_1\tilde \gamma_2'(\xi^{**}),\] a contradiction. This concludes the proof of the monotonicity of $\Gamma_{0,\max}'$.

We now study the asymptotic behaviour of $\Gamma_{0,\max}'$ and, more precisely, we prove 
\begin{equation}\label{Eq:Si5}
\lim_{c \to \infty}\Gamma_{0,\max}'(c)=0\end{equation} as the monotonicity in turn implies $\lim_{c \to 0^+}\Gamma_{0,\max}'(c)>0$. Establishing \eqref{Eq:Si5} follows from a simple analysis of \eqref{Eq:MB}: up to a translation assume that, for any $c>0$, we have 
\[ \Gamma_{0,c}'(0)=\Gamma_{0,\max}'(c).\] As $\Gamma_{0,c}$ satisfies 
\[ -\frac{d}{ds}\left(e^{cs}\Gamma_{0,c}'\right)=e^{cs}f\left(\Gamma_{0,c}(s)\right)\] and as $0\leq \Gamma_{0,c}\leq 1$, an integration between $-\infty$ and 0 provides the upper bound
\[ \Gamma_{0,c}'(0)\leq \frac{\Vert f\Vert_{L^\infty(0;1)}}c.\] The conclusion follows.

\end{proof}

Given Assumption \eqref{Hyp:L}, $L'$ is a non-decreasing function, whereby Lemma~\ref{Le:Galere1} implies that the map $c\mapsto \frac{\Gamma'_{0,\max}(c)}{L'(c)}$ is continuous decreasing. Consequently, it has a left-inverse, which we call $c_0$:
\begin{equation}\label{Eq:c0}c_0:(0;\infty)\ni \lambda \mapsto\left( \frac{\Gamma'_{0,\max}}{L'}\right)^{-1}(\lambda)=\sup\left\{ c>0\,, \Gamma_{0,\max}'(c)\geq \lambda L'(c)\right\}.\end{equation}

\subsection{Construction of the monotonous reversed travelling wave (Proposition~\ref{Prop:exist_th_TW})}
We are now in a position to prove Proposition~\ref{Prop:exist_th_TW}.
\begin{proof}[Proof of Proposition~\ref{Prop:exist_th_TW}]
Recall that we must have $\M$ compactly supported and that $\Th$ must satisfy 
\[ \Th'=\lambda L'(c)\text{ in }\mathrm{supp}(\M).\] Fix $\lambda\,, c>0$ such that $0<c<c_0(\lambda)$ where $c_0$ is defined in \eqref{Eq:c0}. Let $s^*$ be such that $\Gamma_{0,c}'(s^*)=\Gamma'_{0,\max}(c)$. We define $s_0$ as 
\begin{equation}\label{Def:S0} s_0:=\min\left\{s\geq s^*\,, \Gamma_{0,c}'(s_0)=\lambda L'(c)\right\}.\end{equation}  
We first define $\Th$ on $(-\infty;s_0)$ as 
\[ \Th:(-\infty;s_0)\ni s\mapsto \Gamma_{0,c}(s).\] 
We now build $\M$ as a density supported in $[s_0:s_1]$ where $s_1$ has to be determined.
Since we know that $\Th'$ must be constant on $\mathrm{supp}(\M)$ it is natural to construct $(\M,\Th)$, on the support of $\M$, as 
\begin{equation}\label{Eq:Razu}\Th:[s_0;s_1]=\supp(\M)\ni s\mapsto \Gamma_{0,c}(s_0)+s\Gamma_{0,c}'(s_0)\text{ and }\M:=\frac{f(\Th)+c\Th'}\Th.\end{equation} Two requirements need to be checked: the non-negativity of $\M$ and its integrability. Regarding the non-negativity, we use \eqref{Hyp:f2}: let $\underline \eta\in (0;\eta)$ be the unique local minimum of $f$ in $(0;\eta)$.  As $-c  \Gamma_{0,c}'(s^*)=f( \Gamma_{0,c}(s^*))$, we deduce  \begin{equation}\label{Eq:Tau0} \Gamma_{0,c}(s^*)\geq
    \underline \eta.\end{equation} Since $f$ has a unique critical point in $(0;\eta)^{\top}$ (and in particular no critical point in $(\underline \eta;\eta)$) we deduce that for any $s\geq s_0$ such that $\Gamma_{0,c}(s_0)+s\Gamma_{0,c}'(s_0)\leq 1$, there holds 
\[ c\Gamma_{0,c}'(s_0)\geq -f\left(\Gamma_{0,c}(s_0)+s\Gamma_{0,c}'(s_0)\right).
\] Thus, the function $\M$ defined in \eqref{Eq:Razu} satisfies 
\[ \M\geq 0.\] We let $\gamma_1=\Gamma_{1,c}(r_1)$ be the only point such that $(\gamma_1,\Gamma_{0,c}'(s_0))\in S_1=\{(\Gamma_{1,c}(s)\,, \Gamma_{1,c}'(s))\,, s \in \R\}$ for some $r_1$. We define $s_1>s_0$ such that $\Gamma_{0,c}(s_0)+s_1\Gamma_{0,c}'(s_0)=\gamma_1$ and we define, for $s\geq s_1$, 
\[ \Th(s)=\Gamma_1(r_1-s_1+s)\,, \M=0.\]
By construction, $\M$ is compactly supported.
It is clear (since $\Th'\equiv \Gamma_{0,c}'(s_0)$ in $\supp(\M)$) that $\M\in L^1(\R)$. Finally, by construction, $\Th$ satisfies 
\[ \begin{cases}-\Th''-c\Th'=f(\Th)-\M\Th\,,\\ \Th(-\infty)=\Th'(-\infty)=\Th'(+\infty)=0\,,\\ \Th(+\infty)=1,\end{cases}\] and 
\[ \Th'\equiv \lambda L'(c)\text{ in }\mathrm{supp}(\M).\]
\end{proof}

\subsection{Construction of non-monotonous reversed travelling waves (Proposition~\ref{Prop:exist_th_multiple_bumps})}
\label{subsec:prop_multiple_bumps}

\begin{proof}[Proof of Proposition~\ref{Prop:exist_th_multiple_bumps}.]
    The  proof is an adaptation of the proof of Proposition~\ref{Prop:exist_th_TW}. Namely, we consider speed $c$ satisfying $c<2\sqrt{f'(\eta)}$ so that the equilibrium $(\eta,0)^{\top}$ of \eqref{Eq:ODE} is a spiral sink (see Lemma~\ref{Le:Stab2}). In particular, there exists a sequence $(s_{k})_{k\in \N}$ such that $\Gamma_{0,c}'$ has a positive local maximum at $s=s_k$, and has exactly one local minimum (which is negative) between $s_k$ and $s_{k+1}$. To build a reversed travelling wave with exactly $k$ bumps, it suffices to replace, in the proof of Proposition~\ref{Prop:exist_th_TW}, $s^*$ with $s_k$, and the definition of $c_0$ with 
    \[ c_k:(0;+\infty)\ni \lambda\mapsto \sup\left\{c>0\,, \Gamma_{0,c}'(s_k)\geq \lambda L'(c)\right\}.\] 
    The rest of the proof is identical.
    \end{proof}

The following lemma will also be important in the course of the proof of our main results; it is a simple consequence of the construction of $(\Th,\M)$.
\begin{lemma}
    \label{Lem:Th''}
    Let $c\,, \lambda>0$ be two positive real numbers, $k\in \N$ with $c<c_k(\lambda)$, let $(\Th\,, \M)$ be the reversed travelling wave provided by  Proposition~\ref{Prop:exist_th_multiple_bumps}, and recall that up to a translation $\inf(\supp(\M))=0$. Then
$\Th''(0^-)<0$.\end{lemma}

\begin{figure}[h]\begin{center}
\includegraphics[scale=0.15]{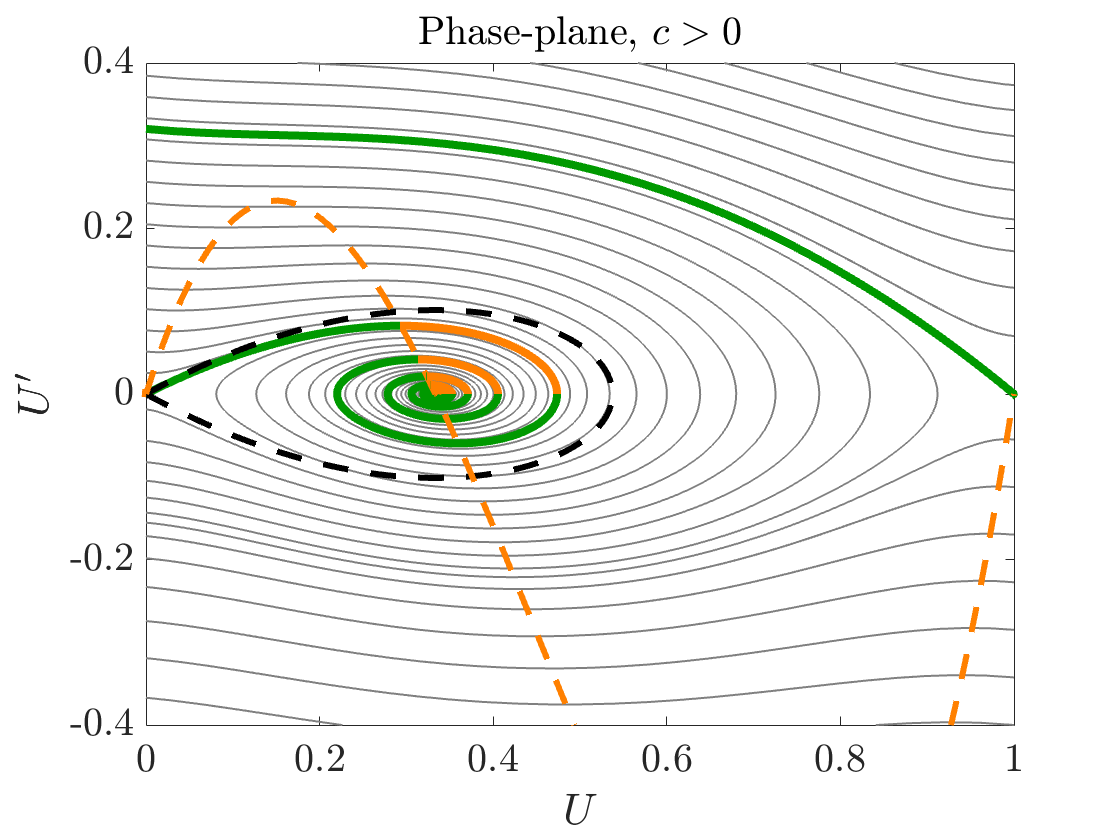}
\end{center}
\caption{Visual representation: we start acting when the dashed orange graph $\Xi$ is left of the (green) unstable manifold $S_0$. The solid orange line represents the points on $S_0$ we might ``jump'' from.}\label{Fig:4}
\end{figure}

\section{The strongly convex setting
    (proof of Theorems~\ref{Th:Quadratique} and~\ref{Th:Quadratique2})}
\label{sec:strongly_convex}
In this section, we prove Theorem~\ref{Th:Quadratique2} (since it suffices to take $k=0$ in Theorem~\ref{Th:Quadratique2} to recover Theorem~\ref{Th:Quadratique}). 
\begin{proof}[Proof of Theorem~\ref{Th:Quadratique2}] We let $k\in \N$. be fixed throughout the proof and we consider a couple $(c,\lambda)$ with $c\in( 0;c_k(\lambda))$, where $c_k$ is given by Proposition~\ref{Prop:exist_th_multiple_bumps}. We let $(\Theta_{\lambda,c},\M_{\lambda,c})$ be the reversed travelling wave provided by the latter proposition with exactly $k$ bumps.

To prove the theorem, it is enough to prove that for any representative player, the constant speed strategy $\overline \alpha\equiv c$ is optimal in the optimisation problem \eqref{Eq:OCPlayer}; this latter optimisation problem here rewrites, the player being, at time $t=0$ in the position $x(0)=x_0\in \mathrm{supp}(\M)$, 
\begin{equation}\label{Eq:OCPlayerQuadratique}
\sup_{\alpha\in L^\infty(\R)}J(x_0,\alpha):=\int_0^\infty e^{-\lambda t}\left(\Th_c(x_\alpha(t)-ct)-L(\alpha(t))\right)dt\text{ subject to }\begin{cases}\frac{dx_\alpha}{dt}=\alpha\,, \\ x_\alpha(0)=x_0.\end{cases}
\end{equation}
To this end, we use a concavity approach, first showing that, for any $c\,, \lambda$ with $0<c<c_k(\lambda)$, $\overline\alpha\equiv c$ is a critical\footnote{Criticality is to be understood as: for any perturbation $h\in L^\infty$, the Gateaux derivative of $J_{\lambda,c}$ at $\overline \alpha$ in the direction $h$ is zero.} point of $J_{\lambda,c}$ (Lemma~\ref{Le:Quadratique1}), then proving that, for $\lambda>0$ small enough, $J_{\lambda,c}$ is strictly concave in $\alpha$ (Lemma~\ref{Le:Quadratique2}).
 Throughout, we let $\overline \alpha$ be the constant control:
\[ \overline \alpha\equiv c.\]
\begin{lemma}\label{Le:Quadratique1}
For any $(c,\lambda)$ satisfying $0<c<c_k(\lambda)$, $\overline \alpha$ is a critical point of $J_{c,\lambda}$.
\end{lemma}
\begin{proof}[Proof of Lemma~\ref{Le:Quadratique1}]
Observe that when $\alpha=\overline \alpha$ we have $\Th_{\lambda,c}(x_\alpha(t)-ct)=\Th(x_0)$.
The Gateaux differentiability of $\alpha\mapsto J(x_0,\alpha)$ is immediate and, for any $\alpha\,, h\in L^\infty$, the Gateaux derivative of $J(x_0,\cdot)$ at $\alpha$ in the direction $h$ is given by 
\begin{align*}
\Dot J(x_0, \alpha)[h]&=\int_0^\infty e^{-\lambda t}\left(\Th_{\lambda,c}'(x_\alpha(t)-ct)\left(\int_0^t h\right)-L'(\alpha)h\right)dt
\end{align*}
We introduce the adjoint state $p_{\lambda,c,\alpha}$ as the unique solution of the equation 
\[\begin{cases}
-p_{\lambda,c,\alpha}'=e^{-\lambda t}\Th_{\lambda,c}'(x_\alpha(t)-ct)\,, 
\\ p_{\lambda,c,\alpha}(+\infty)=0.
\end{cases}\]
Thus $\dot J(x_0,\alpha)[h]$ rewrites as 
\[ \dot J(x_0,\alpha)[h]=\int_0^\infty \left(p_{\lambda,c,\alpha}-L'(\alpha)e^{-\lambda t}\right)h.
\]
The first-order optimality condition is then:
\[ p_{\lambda,c,\alpha}=L'(\alpha).\] Let us check that this condition is satisfied for $\alpha=\overline\alpha$. 

If $\alpha=\overline\alpha$, then $\Th_{\lambda}(x_\alpha(t)-ct)=\Th_{\lambda}(x_0)$, whence we deduce that 
\[ p_{\lambda,c,\alpha}(t)=\frac{1}\lambda e^{-\lambda t} \Th_{\lambda,c}'(x_0).\] Consequently, we are left with verifying that 
\[ \Th_{\lambda,c}'(x_0)=\lambda L'(c),\] which follows by our choice of travelling wave.
\end{proof}
\begin{lemma}\label{Le:Quadratique2} For any $k\in \N$, there exists a constant $\lambda_k(\underline D)$ such that for any $\lambda>0$ satisfying 
\[\tilde M<\lambda\]
the map $J_{\lambda,c}$ is strictly concave in $\alpha$.
\end{lemma}
\begin{proof}[Proof of Lemma~\ref{Le:Quadratique2}]
Similar to the proof of Lemma~\ref{Le:Quadratique1}, note that the second-order Gateaux derivative of $J_{\lambda,c}$ at any $\alpha$ in a direction $h$ writes 
\[ \ddot J(x_0,\alpha)[h,h]=\int_0^\infty e^{-\lambda t}\left(\Th_{\lambda,c}''(x_\alpha(t)-ct)\left(\int_0^t h\right)^2-L''(\alpha)h^2\right)dt.\]
As we know from the equation satisfied by $\Th$, and from the fact that $c_k\leq 2\sqrt{f'(\eta)}$, that  $\Th_{\lambda,c}''$ is bounded uniformly, there exists a constant $M_k>0$ such that 
\[ \forall \lambda \in (0;\infty)\,, \forall c \in (0;c_k(\lambda))\,, \Vert\Th_{c,\lambda}''\Vert_{L^\infty}\leq  M_k.\]
Thus we have the following upper bound on the second-order derivative of $J_{\lambda,c}$: for any $\lambda>0$, for any $c\in (0;c_k(\lambda))$, for any $\alpha$ and any $h$,
\[ \ddot J(x_0,\alpha)[h,h]\leq M_k\int_0^\infty e^{-\lambda t}\left(\int_0^t h\right)^2dt-\underline D\int_0^\infty e^{-\lambda t}h^2(t)dt.\]
For any $h\in L^\infty$ define 
\[ v_h:t\mapsto \int_0^t h\] so that the previous inequality rewrites 
\[ \ddot J(x_0,\alpha)[h,h]\leq M_k\int_0^\infty e^{-\lambda t}v_h^2(t)dt-\underline D\int_0^\infty e^{-\lambda t}(v_h')^2(t)dt.
\]
However, observe that, for any $v\in W^{1,\infty}(\R)$ satisfying $v(0)=0$, we have 
\begin{align*}
\int_0^\infty e^{-\lambda t} v^2(t)dt&=\frac2\lambda \int_0^\infty e^{-\lambda t} v(t)v'(t)dt
\\&\leq \frac2\lambda \sqrt{\int_0^\infty e^{-\lambda t} v^2(t)dt}\sqrt{\int_0^\infty e^{-\lambda t}(v')^2(t)dt}
\end{align*}
whence 
\[ \int_0^\infty e^{-\lambda t} v^2(t)dt\leq \frac{4}{\lambda^2}\int_0^\infty e^{-\lambda t} (v')^2(t)dt.\] Thus we conclude that 
\[ \ddot J(x_0,\alpha)[h,h]\leq \left(\frac{4 M_k}{\lambda^2}-\underline D\right)\int_0^\infty e^{-\lambda t} (v')^2(t)dt.\]
Hence, if $\lambda>0$ is chosen so that 
\[ 4 M_k<\underline D \lambda^2\] (and $ M_k\,, \underline D$ do not depend on $\lambda$), we deduce that $\ddot J_{\lambda,c}$ is concave in $\alpha$. Consequently, any critical point is optimal, whence $\overline\alpha$ is optimal for \eqref{Eq:OCPlayerQuadratique}.
\end{proof}
\end{proof}

\section{The general setting (proof of Theorems~\ref{Th:Main}-\ref{Th:small_lambda})}
\label{sec:general_case}From Lemma~\ref{Le:MonotComp} we know that, in order for any reversed travelling wave $(\Th,\M)$  to be a reversed MFG travelling wave, $\Th$ has to be affine in the support of $\M$. The proofs of Theorems~\ref{Th:Main}-\ref{Th:small_lambda} both rely on the same core idea of comparing the travelling wave $(\Th,\M)$ provided by Proposition~\ref{Prop:exist_th_TW} or~\ref{Prop:exist_th_multiple_bumps} with a linear function $\tilde \Th$. This can be seen as some form of linearisation of the optimal control problem.
\subsection{Proof of Theorem~\ref{Th:Main}}

\paragraph{A preliminary result}
Therefore, let $\lambda>0$ and $0<c<c_0(\lambda)$, and consider the reversed travelling wave  $(\Th,\M)$  provided by Proposition \ref{Prop:exist_th_TW}. 
We define the linear extension of $\Th$ outside of $\supp(\M)$ as 
\begin{equation}
    \label{Eq:def_Tht}
    \Tht:s\mapsto(\lambda L'(c))s+\Theta(0),
\end{equation}
We introduce the auxiliary optimal control problem
\begin{equation*}
    \begin{aligned}
        &\cVt(s_0)
        =
        \sup_{\alpha\in L^\infty(\R,\R)}
        \cJt_0(s_0,\alpha)
        \\
        &\text{ where for any $t_0\geq 0$ }\quad \cJt_{t_0}(s_0,\alpha)
        =
        \left(\int_{t_0}^\infty e^{-\lambda t}\left(\Tht(s_{\alpha}(t))
                -L(\alpha)\right)dt\right)
        \text{ with }
        \begin{cases}
            s_\alpha'(t)=\alpha(s_\alpha(t))-c\,,
            \\
            s_\alpha(t_0)=s_0,\end{cases}
    \end{aligned}
\end{equation*}
\begin{lemma}\label{Le:OptimalStratTilde}
    Assume that $L$ is convex and $\mathscr C^1$.
    The constant control function $\overline \alpha\equiv c$
    is optimal in the definition of $\cVt$. Furthermore,
    for any $s_0\in \R$, 
    \[ \cVt(s_0)
        =
        \frac{\Tht(s_0)-L(c)}\lambda\,,
        \quad
        \text{ whereby }
        \quad
        \cVt'\equiv L'(c).\]
\end{lemma}
\begin{proof}
    Observe that 
    $s\mapsto\Tht(s)$ and
    $\alpha\mapsto s_\alpha$ are linear,
    so that the map 
    \[ \alpha \mapsto \int_0^\infty e^{-\lambda \xi}\tilde\Th(s_\alpha(\xi))d\xi\]
    is linear as well.
    Since $L$ is convex in $\alpha$, 
    $\Jt$ is concave in $\alpha$.
    Therefore,
    to prove that $c$ is an optimal control,
    it is sufficient to prove that it is a critical point of $\Jt(s_0,\cdot)$
    which follows by exactly the same arguments as in Lemma \ref{Le:Quadratique1}.
\end{proof}

The proof of Theorem \ref{Th:Main} rests upon the following comparison lemma:
\begin{lemma}
    \label{Lem:compare_Tht}
  Let $\lambda>0\,, c\in (0;c_0(\lambda))$ and $(\Th,\M)$ be provided by 
    Proposition \ref{Prop:exist_th_TW}. Let $\Tht$ defined by \eqref{Eq:def_Tht}.
      For any $s_0\in\R$, $T>0$
    and any $\alpha\in L^\infty(\R)$,
there holds    \begin{align*}
        \cVt(s_0)
        -\mathcal{J}_0(s_0,\alpha)
               &\geq
        e^{-\lambda T}
        \left(
            \cVt(s_\alpha(T))
            -\mathcal{J}_T(s_\alpha(T),\alpha)
        \right)
        +\int_0^T
        e^{-\lambda t}
        \left(
            \Tht(s_\alpha(t))
            -\Th(s_\alpha(t))
        \right)\,dt,
    \end{align*}
    where $\frac{d}{dt}s_\alpha(t)=\alpha(s_\alpha(t))-c$
    and $s_\alpha(0)=s_0$.
\end{lemma}

\begin{proof}[Proof of Lemma \ref{Lem:compare_Tht}]
    Using the fact that the constant
    control $c$ leads to a constant trajectory,
    and the definition of $J$,
    we obtain 
    \begin{equation*}
        \cVt(s_0)
        -\mathcal{J}_0(s_0,\alpha)
        =
        e^{-\lambda T}
        \left(
            \cVt(s_0)
            -\mathcal{J}_T(s_\alpha(T),\alpha)
        \right)
        +\int_0^T
        e^{-\lambda t}
        \Bigl(
            \Tht(s_0)
            -L(c)
            -\Th(s_\alpha(t))
            +L(\alpha(s_\alpha(t)))
        \Bigr)\,dt.
    \end{equation*}
    Recall that
    $\Tht(s_0)
    =\Tht(s)-\lambda L'(c)(s-s_0)$,
    we get that
    \begin{align*}
        \int_0^Te^{-\lambda t}
        \Tht(s_0)\,dt
        &=
        \int_0^Te^{-\lambda t}
        \Bigl(
        \Tht(s_\alpha(t))
        -\lambda L'(c)(s_\alpha(t)-s_0)
        \Bigr)\Th(s_0)\,dt
        \\
        &=
        \int_0^Te^{-\lambda t}
        \Tht(s_\alpha(t))\,dt
        +\Bigl[
        e^{-\lambda t}
        L'(c)(s_\alpha(t)-s_0)
        \Bigr]^T_0
        -\int_0^T
        e^{-\lambda t}
        L'(c)(\alpha(s_\alpha(t))-c)
        \,dt
        \\
        &=
        \int_0^Te^{-\lambda t}
        \Tht(s_\alpha(t))\,dt
        +e^{-\lambda T}
        L'(c)(s_\alpha(T)-s_0)
        -\int_0^T
        e^{-\lambda t}
        L'(c)(\alpha(s_\alpha(t))-c)
        \,dt.
    \end{align*}
    Recall that
    $\cVt(s)=\cVt(s_0)+L'(c)(s-s_0)$,
    we obtain
    \begin{align*}
        \cVt(s_0)
        -\mathcal{J}_0(s_0,\alpha)
        &\!\begin{multlined}[t][10.5cm]
        =
        e^{-\lambda T}
        \left(
            \cVt(s_\alpha(T))
            -\mathcal{J}_T(s_\alpha(T))
        \right)
        +\int_0^T
        e^{-\lambda t}
        \left(
            \Tht(s_\alpha(t))
            -\Th(s_\alpha(t))
        \right)\,dt
        \\
        +\int_0^T
        e^{-\lambda t}
        \Bigl(
            L(\alpha(s_\alpha(t)))
            -L(c)
            -(\alpha(s_\alpha(t))-c)L'(c)
        \Bigr)\,dt
        \end{multlined}
        \\
        &\geq
        e^{-\lambda T}
        \left(
            \cVt(s_\alpha(T))
            -\mathcal{J}_T(s_\alpha(T),\alpha)
        \right)
        +\int_0^T
        e^{-\lambda t}
        \left(
            \Tht(s_\alpha(t))
            -\Th(s_\alpha(t))
        \right)\,dt,
    \end{align*}
    where the last inequality is obtained
    using the convexity of $L$.
\end{proof}

We are now in a position to prove Theorem \ref{Th:Main}.
\begin{proof}[Proof of Theorem \ref{Th:Main}]

\label{subsec:proof_Th_Main}
For a fixed discount factor $\lambda>0$,
the proof consists in proving
that the constant control function~$\overline \alpha\equiv c$
is optimal in \eqref{Eq:OCPlayer} whenever $c$ is small enough. To be more precise, we are going to show that,
    for $c$ small enough,
    for any $s_0\in\supp(\M)$,
    \begin{equation}
        \label{Eq:aux_Th_Main}
        \Jt_0(s_0,c)
        -\sup_{\alpha\in L^{\infty}} J_0(s_0,\alpha)
        \geq
        0.
    \end{equation}
 The proof of this fact relies on a comparison with the auxiliary problem defined by $\tilde\cV$. Lemma \ref{Lem:compare_Tht} is an
    an essential argument in the forthcoming analysis.

We let $\lambda>0$ and $c\in (0;c_0(\lambda))$
and we consider the couple
$(\Th,\M)$ provided by
Proposition \ref{Prop:exist_th_TW}. Up to a translation, we assume that $\inf(\supp(\M))=0$.
Define $\Tht$ by \eqref{Eq:def_Tht}
with $\Tht(0)=\Th(0)$.

\paragraph{Intersection between $\Th$ and $\Tht$}
Define
\[ s_-(\lambda,c):=\sup\{s\leq 0\,, \tilde \Th(s)<\Th(s)\}.\]
 As by construction $\Th''_-(0)<0$, we have 
 \[ s_-(\lambda,c)<0.\] 
 
\begin{figure}
\includegraphics[scale=0.2]{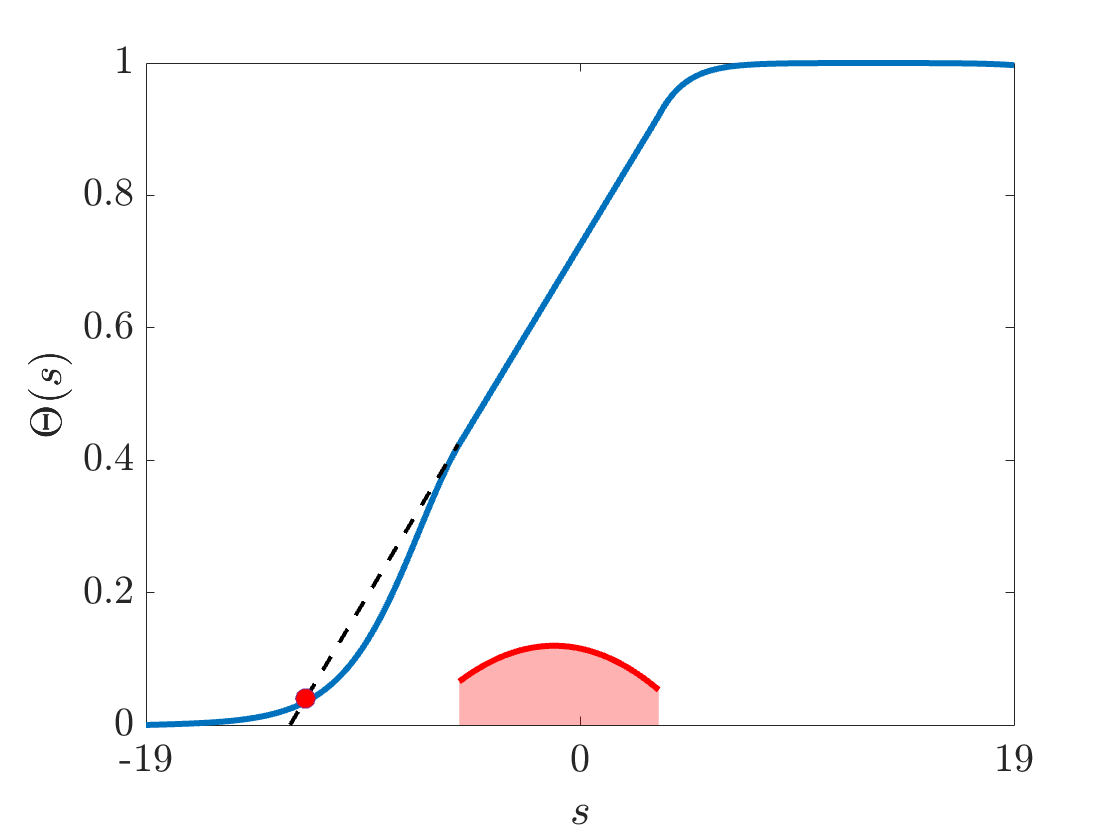}
\includegraphics[scale=0.2]{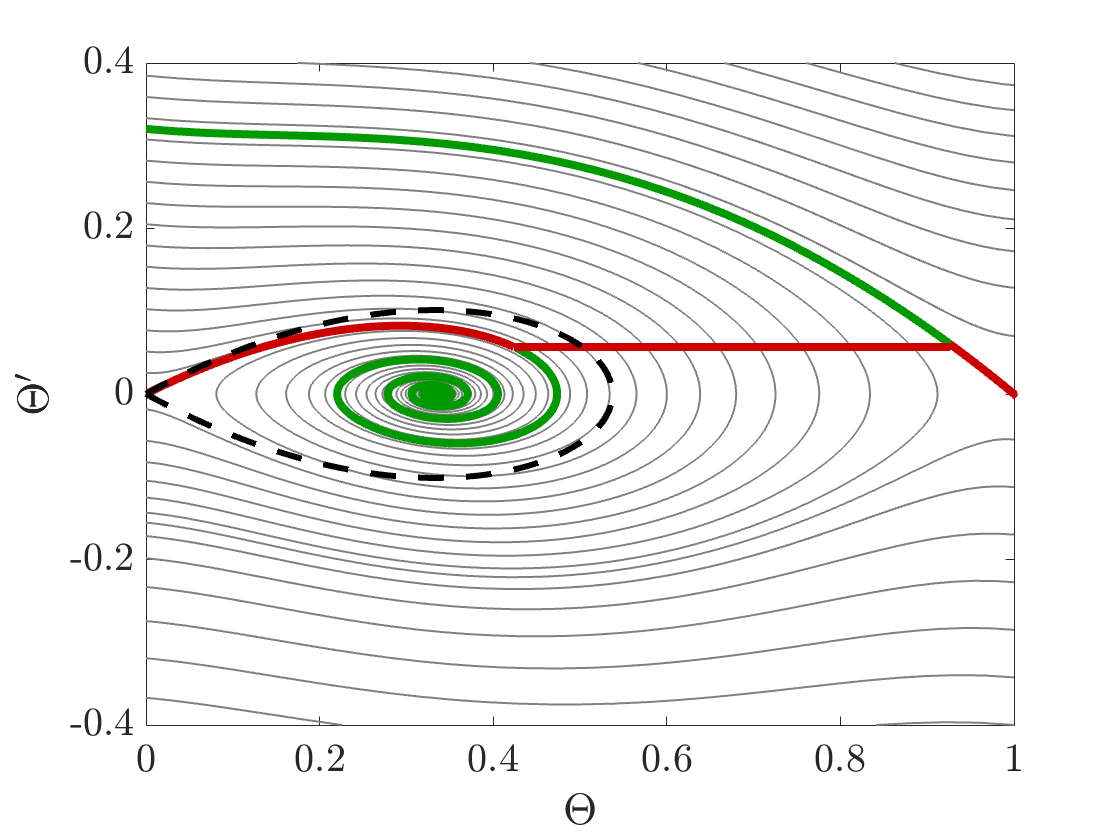}
\caption{(Left) Reversed traveling wave $\Theta$ (in blue) and $\mathcal{M}$ (the subgraph in red); the dashed black line has the same slope as $\Theta$ in the support of $\mathcal{M}$ and the red point represents $s_-(\lambda,c)$; (right) The corresponding phase-plane representation of the traveling wave represented on the left.}\label{Fig:slambdac}
\end{figure}

Furthermore, it follows by construction that
\begin{equation}
    \label{Eq:def_s-}
    \Tht(s)<\Th(s)
    \,\text{ for }\,
    s< s_-(\lambda,c)
    \quad\text{ and }\quad
    \Tht(s)\geq\Th(s)
    \,\text{ for }\,
    s\geq s_-(\lambda,c).
\end{equation}
 We also note, for further references, the two following obvious properties of $s_-$:
 \begin{equation}\label{Eq:S-}
 s_-(\lambda,c)\underset{c \to 0}\rightarrow -\infty\,, \Th(s_-(\lambda,c))\underset{c\to 0}\rightarrow 0.
 \end{equation}
 This is merely a consequence of the construction of $(\Th,\M)$. Indeed, as $c\to 0$, $\Th'(c)=\lambda L'(c)\underset{ c\to 0}\to 0$ as well and $\inf(\supp(\M))=s_{00}$ where $s_{00}$ is the unique local maximiser of the curve $\Gamma_{0,0}$, the unstable manifold of $\begin{pmatrix}0\\0\end{pmatrix}$ in \eqref{Eq:ODE} when $c=0$.

\paragraph{Restriction of the class of admissible controls} Our goal is to prove that, for any $s_0\in \supp(\M)$, we can restrict ourselves to studying controls $\alpha$ such that 
\begin{equation}\label{Eq:PositivControl}\alpha\geq 0\text{ a.e., }\end{equation} and that satisfy
\begin{equation}\label{Eq:ReachedIntersection} \exists T_1=T_1(\alpha)\,, s_\alpha(T_1)=s_-(\lambda,c).\end{equation}
First of all, $\Th$ is an increasing function and $0$ is a minimum of the Lagrangian $L$, whence, for any $\alpha$,  
\[\forall s_0\in \mathrm{supp}(\mathcal{M})\quad \mathcal{J}_0(s_0,\alpha_+)\geq \mathcal{J}_0(0,\alpha).\] In particular, we can assume that \eqref{Eq:PositivControl} is satisfied. Second, assume \eqref{Eq:ReachedIntersection} does not hold for an optimal control $\alpha$. In that case, the entire trajectory $s_\alpha$ satisfies $s_\alpha>s_-(\lambda,c)$, whence 
\[ \mathcal{J}_0(s_0,\alpha)\leq \cJt_0(s_0,\alpha)\leq\cJt_0(s_0,c)=\mathcal{J}_0(s_0,c),\] and we conclude that $\overline\alpha\equiv c$ is optimal in \eqref{Eq:OCPlayer}.  Thus the only case that remains to be treated is that of controls $\alpha$ such that \eqref{Eq:ReachedIntersection} holds. As a consequence of \eqref{Eq:ReachedIntersection}, we further deduce that we can restrict ourselves to controls $\alpha$ satisfying 
\begin{equation}\label{Eq:BoundedControls1}
0\leq \alpha< c\text{ in }(s_-(\lambda,c);s_0).\end{equation} 
Finally, this implies that, up to replacing with $\min(\alpha,c)$, we can choose $\alpha$ so that
\begin{equation}\label{Eq:BoundedControls}
0\leq \alpha< c\text{ in }\R.\end{equation} 

    Consequently,
    it is sufficient to prove
    \eqref{Eq:aux_Th_Main}
    when the supremum is taken over
    control functions
    $\alpha$ with values
    in $[0,c)$,
    and such that
    there exists $T_1(\alpha)>0$
    satisfying $s_\alpha(T_1(\alpha))=s_-(\lambda,c)$ (and $T_1(\alpha)$ is chosen as the first root of this equation).

 In this case, and since $s_\alpha$ is decreasing in $[0;T_1(\alpha))$,
    we can define
    $T_{\frac12}(\alpha)\in (0;T_1(\alpha))$ as the unique root of the equation $s_\alpha(T_{\frac12}(\alpha))=\frac{s_-(\lambda,c)}{2}$.
    Recall that $J(s_0,c)=\cVt(s_0)$ since $s_0\in\supp(\M)$.
    Using Lemma \ref{Lem:compare_Tht}
    with  $T=T_{\frac12}(\alpha)$,
    and using the inequality
    $\Tht(s_\alpha(t))
    \geq\Th(s_\alpha(t))$
    for $t\leq T_1(\alpha)$,
    we obtain
    \begin{equation}\label{Ineq1}
        \mathcal{J}_0(s_0,c)
        -\mathcal{J}_0(s_0,\alpha)=\cVt(s_0)-J_0(s_0,\alpha)
        \geq
        e^{-\lambda T_{\frac12}(\alpha)}
        \Bigl(
        \cVt\left(s_\alpha\left(T_{\frac12}(\alpha)\right)\right)
        -J_{T_{\frac12}(\alpha)}\left(s_\alpha\left(T_{\frac12}(\alpha)\right),\alpha\right)
        \Bigr).
    \end{equation}
    Using once again Lemma \ref{Lem:compare_Tht}, $s_0$ being replaced with $s_\alpha\left(T_{\frac12}(\alpha)\right)$ and $T$ with $T_1(\alpha)-T_{\frac12}(\alpha)$,
    we get
    \begin{multline}\label{Ineq2}
        \cVt\left(s_\alpha\left(T_{\frac12}(\alpha)\right)\right)
        -\mathcal{J}_{T_{\frac12}(\alpha)}\left(s_\alpha\left(T_{\frac12}(\alpha)\right),\alpha\right)
        \geq
        e^{-\lambda\left(T_1(\alpha)-T_{\frac12}(\alpha)\right)}
        \Bigl(
        \cVt(s_\alpha(T_{1}(\alpha)))
        -\mathcal{J}_{T_{1}(\alpha)}(s_\alpha(T_{1}(\alpha)),\alpha)
        \Bigr)
        \\
        +\int_{T_{\frac12}(\alpha)}^{T_1(\alpha)}
        e^{-\lambda \left(t-T_{\frac12}(\alpha)\right)}
        \Bigl(\Tht(s_\alpha(t))
        -\Th(s_\alpha(t))
        \Bigr)\,dt.
    \end{multline}
    Using $\alpha<c$,
    it is straight forward to check that
 \begin{equation}\label{Ineq3}
        \cVt(s_\alpha(T_{1}(\alpha)))
        -\mathcal{J}_{T_{1}(\alpha)}(s_\alpha(T_{1}(\alpha)),\alpha)
        \geq
        -\lambda^{-1}L(c).
    \end{equation}
    Indeed,
    \begin{align*}
       \cVt(s_\alpha(T_{1}(\alpha)))
        -\mathcal{J}_{T_{1}(\alpha)}(s_\alpha(T_{1}(\alpha)),\alpha)&=\int_0^\infty e^{-\lambda t}\left(\Tht (s_\alpha(T_1(\alpha)))-\Th(s_\alpha(T_1(\alpha)+t))\right)dt
        \\&+\int_0^\infty e^{-\lambda t}\left(-L(c)+L(\alpha(s_\alpha(T_1(\alpha))+t)\right)dt
        \\&-\int_0^\infty e^{-\lambda t} L(c)dt=        -\lambda^{-1}L(c).
   \end{align*}
   We used the fact that $\Th$ is increasing, that $L\geq 0$ and that $\alpha<c$.

    Combining \eqref{Ineq1}-\eqref{Ineq2}-\eqref{Ineq3} 
    we obtain
    \begin{equation}
        \label{Eq:ineq_J_T12}
        e^{\lambda T_{\frac12}(\alpha)}
        (\mathcal{J}_0(s_0,c)
        -\mathcal{J}_0(s_0,\alpha))
        \geq
        -\lambda^{-1}L(c)
        e^{-\lambda\left(T_1(\alpha)-T_{\frac12}(\alpha)\right)}
        +\int_{T_{\frac12}(\alpha)}^{T_1(\alpha)}
        e^{-\lambda\left(t-T_{\frac12}(\alpha)\right)}
        \Bigl(\Tht(s_\alpha(t))
        -\Th(s_\alpha(t))
        \Bigr)\,dt.
    \end{equation}
We now provide an $\alpha$-independent lower bound on the right-hand side of \eqref{Eq:ineq_J_T12} when $c$ is small enough.
    We start by bounding by below the exponent
        $\lambda
        \left(T_1(\alpha)
        -T_{\frac12}(\alpha)\right)$.
    The inequality $\alpha\geq0$ implies
    \begin{equation*}
        \frac{s_-(\lambda,c)}2
        =
        s_\alpha(T_1(\alpha))
        -s_\alpha(T_{\frac12}(\alpha))
        \geq
        -c\left(T_1(\alpha)
        -T_{\frac12}(\alpha)\right)
    \end{equation*}
    which leads to
    \begin{align*}
        \frac{\Th(0)-\Th(s_-(\lambda,c))}2
        =
        \Tht(s(T_{\frac12}(\alpha)))
        -\Tht(s(T_1(\alpha)))
        &=
        \lambda L'(c)
        \left(s(T_{\frac12}(\alpha))
        -s(T_1(\alpha))\right)
        \\
        &\leq
        \lambda c L'(c)
        \left(T_1(\alpha)
        -T_{\frac12}(\alpha)\right),
    \end{align*}
whence
    \begin{equation*}
        \lambda
        (T_1(\alpha)
        -T_{\frac12}(\alpha))
        \geq
        \frac{\Th(0)-\Th(s_-(\lambda,c))}{2cL'(c)}.
    \end{equation*}
    One may notice that the right-hand side
    of the latter inequality is independent
    of $\alpha$, but depends on $\lambda$ and $c$.
    Moreover, recall that from \eqref{Eq:S-}
$        s_-(\lambda,c)
        \underset{c \to 0}\rightarrow
        -\infty\,,
        \Th(s_-(\lambda,c))\underset{c\to 0}\rightarrow 0.
$    
Now, on the one hand,
    we have
    \begin{equation*}
        \lambda^{-1}L(c)
        e^{-\lambda(T_1(\alpha)-T_{\frac12}(\alpha))}
        \leq
        \lambda^{-1}L(c)
        e^{-\frac{\Th(0)-\Th(s_-(\lambda,c))}{2cL'(c)}},
    \end{equation*}
    where the right-hand side is independent of $\alpha$
    and converges to $0$  as $c\to 0$.
    
    On the other hand,
    for $c$ small enough,
    we have 
    $T_1(\alpha)-T_{\frac12}(\alpha)\geq1$. As $s_\alpha$ is decreasing, as $\Th$ and $\Tht$ are increasing, the inequality $e^{-x}(1+x)\leq 1\, (\text{for } x\geq 0)$ implies
    \begin{align*}
        \int_{T_{\frac12}(\alpha)}^{T_1(\alpha)}
        e^{-\lambda(t-T_{\frac12}(\alpha))}
        \Bigl(\Tht(s_\alpha(t))
        -\Th(s_\alpha(t))
        \Bigr)\,dt
        &\geq
        \int_{T_{\frac12}(\alpha)}^{T_{\frac12}(\alpha)+1}
        e^{-\lambda(t-T_{\frac12}(\alpha))}
        \Bigl(\Tht(s_\alpha(t))
        -\Th(s_\alpha(t))
        \Bigr)\,dt
        \\
        &\geq
        e^{-\lambda}
        \left(\Tht(s_\alpha(T_{\frac12}+1))
        -\Th(s_\alpha(T_{\frac12}))\right)
        \\
        &\geq
        e^{-\lambda}
        \left(
        \Tht\left(\frac{s_-(\lambda,c)}2\right)
        -c\lambda L'(c)
        -\Th\left(\frac{s_-(\lambda,c)}2\right)
        \right).
    \end{align*}
    Observe that
    the right-hand side of the last
    inequality is independent of $\alpha$
    and converges, as $c\to 0$, to        $e^{-\lambda}
        \frac{\Th(0)}2>0$.

    Combining the latter inequalities with
    \eqref{Eq:ineq_J_T12}, we obtain
    \begin{align*}
        e^{\lambda T_{\frac12}(\alpha)}
        (\mathcal{J}_0(s_0,c)
        -\mathcal{J}_0(s_0,\alpha))
        &\geq
        -\lambda^{-1}L(c)
        e^{-\frac{\Th(0)-\Th(s_-(\lambda,c))}{2cL'(c)}}
        +e^{-\lambda}
        \left(
        \Tht\left(\frac{s_-(\lambda,c)}2\right)
        -c\lambda L'(c)
        -\Th\left(\frac{s_-(\lambda,c)}2\right)
        \right)
        \\
        &>
        0,
    \end{align*}
    when $c$ is small enough, i.e.,
    $c\leq c_1(\lambda)$ for some
    $c_1(\lambda)>0$.
    This implies that
    $\mathcal{J}_0(s_0,c)\geq V(s_0)$
    leading to 
    $\mathcal{J}_0(s_0,c)= V(s_0)$
    and the fact that the constant
    control $c$ is optimal.
    Therefore,
    $(c,\Vt,\M,\Th)$
    is indeed a
    monotonous reversed MFG travelling wave
    in the sense of Definition \ref{De:ReversedMFGTW}.
\end{proof}

\subsection{Solutions with maximal velocities (proof of Theorem \ref{Th:small_lambda})}
\label{subsec:proof_small_lambda}
The core ideas to prove Theorem \ref{Th:small_lambda} are similar to those that were used when dealing with Theorem \ref{Th:Main}. Recall that \eqref{Eq:S-} was an important step; here, we will rather take $c\in (0;c_{\max})$ and let $\lambda \to 0$. While we will observe that there still holds $s_-(\lambda,c)\to-\infty$ as $\lambda\to 0$, we need to be much more careful when handling the different estimates required.

\begin{proof}[Proof of Theorem \ref{Th:small_lambda}]
    Here, we consider a fixed $c\in(0,c_{\max})$; the discount factor $\lambda$ is first chosen small enough to ensure that $c_0(\lambda)\geq c_{\max}$, and will be adjusted in the course of the proof.    We still work, up to a translation, with $0=\inf(\supp(\M))$. Recall that $c<c_{\max}$
    implies
    $\Vert \Gamma_{0,c}\Vert_{L^\infty}>L(c)$,
    so there exists
    $r\in(0,1)$
    such that
    $(1-3r)\Vert \Gamma_{0,c}\Vert_{L^\infty}>L(c)$.
    Recall that $\Tht_c(s_-(\lambda,c))=\Th_c(0)+s_-(\lambda,c)\lambda L'(c)$, and  that $\lambda\mapsto \Th_c(0)$
    is a decreasing function which converges
    to $\Vert\Gamma_{0,c}\Vert_{L^\infty}$ when $\lambda$
    goes to zero.

    This implies that
    \begin{equation*}
        -\lambda s_-(\lambda,c)
        \underset{\lambda\to0}\rightarrow
        \frac{\Vert \Gamma_{0,c}\Vert_{L^\infty(\R)}}{L'(c)}.
    \end{equation*}
        Therefore there exists
    $\lambda_r(c)>0$ such that
    for all $\lambda<\lambda_r(c)$,
    we have
    \begin{equation}\label{Eq:Yu}
        \Th_c(0)\geq(1-r)\Vert \Gamma_{0,c}\Vert_{L^\infty}
        \quad
        \text{ and }
        \quad
        -\lambda s_-(\lambda,c)
        \leq
        2\frac{\Vert \Gamma_{0,c}\Vert_{L^\infty}}{L'(c)}.
    \end{equation}
Henceforth, we always assume that $\lambda<\lambda_r(c)$.
    
    Similar arguments as
    in the proof of Theorem
    \ref{Th:Main} in Section
    \ref{subsec:proof_Th_Main}
    imply that
    it is sufficient to prove
    \eqref{Eq:aux_Th_Main}
    where the supremum is taken
    over control functions
    $\alpha$ with $0\leq\alpha<c$
    and such that $s_-(\lambda,c)$
    is reached at some time
    $T_1(\alpha)>0$,
    where $s_-(\lambda,c)$ is defined
    by \eqref{Eq:def_s-}.
    Let us define $T_r(\alpha)>0$
    such that
    $s_\alpha(T_r(\alpha))=rs_-(\lambda,c)$.
    Using Lemma \ref{Lem:compare_Tht} from
    $s_0\in\supp(\M)$ with $T=T_r(\alpha)$,
    the equality $\mathcal{J}_0(s_0,c)=\cVt(s_0)$
    and the fact that
    $\Tht(s_\alpha(t))
    \geq \Th(s_\alpha(t))$
    for $t\leq T_r(\alpha)$,
    we obtain
    \begin{equation*}
        \mathcal{J}_0(s_0,c)
        -\mathcal{J}_0(s_0,\alpha)
        \geq
        e^{-\lambda T_r(\alpha)}
        \left(\cVt(s_\alpha(T_r(\alpha))
        -J_{T_r(\alpha)}(s_\alpha(T_r(\alpha)),\alpha)\right).
    \end{equation*}
    On the one hand,
    we have
    \begin{align*}
        \lambda\cVt(s_\alpha(T_r(\alpha)))
        &=
        \Tht(s_\alpha(T_r(\alpha)))
        -L(c)
        \\
        &=
        \Th(0)
        +s_\alpha(T_r(\alpha))\lambda L'(c)
        -L(c)
        \\
        &=
        \Th(0)
        +rs_-(\lambda,c)\lambda L'(c)
        -L(c)
        \\
        &\geq
        (1-r)
        \Th_{0,\max}(c)
        +2r
        \Th_{0,\max}(c)
        -L(c)
        >
        0,
    \end{align*}
    where the first line comes
    from the characterization of $\cVt$
    in Lemma \ref{Le:OptimalStratTilde},
    the second from
    the linearity of $\Tht$,
    the third from
    the definition of $T_r(\alpha)$
    and the last from
    the definition of $r$ and \eqref{Eq:Yu}.
  
    On the other hand,
    using the definitions
    of $\mathcal{J}$ and $T_r(\alpha)$,
    the facts that $L$ is nonnegative
    and that $\Th\circ s_\alpha$
    is decreasing,
    we get
    \begin{align*}
        \lambda 
        \mathcal{J}_{T_r(\alpha)}(s_\alpha(T_r(\alpha)),\alpha)
        &\begin{multlined}[t][10.5cm]
        =
        \lambda\int_{T_r(\alpha)}^{T_1(\alpha)}
        e^{-\lambda (t-T_r(\alpha))}
        \left(
            \Th(s_\alpha(t))
            -L(\alpha(s_\alpha(t)))
        \right)\,dt
        \\
        +\lambda\int_{T_1(\alpha)}^{\infty}
        e^{-\lambda (t-T_r(\alpha))}
        \left(
            \Th(s_\alpha(t))
            -L(\alpha(s_\alpha(t)))
        \right)\,dt
        \end{multlined}
        \\
        &\leq
        \lambda\int_{T_r(\alpha)}^{T_1(\alpha)}
        e^{-\lambda (t-T_r(\alpha))}
        \left(
            \Th(s_\alpha(T_r(\alpha)))
        \right)\,dt
        +\lambda\int_{T_1(\alpha)}^{\infty}
        e^{-\lambda (t-T_r(\alpha))}
        \Th(s_\alpha(T_1(\alpha)))
        \,dt
        \\
        &\leq
        (1
        -e^{-\lambda (T_1(\alpha)-T_r(\alpha))})
        \Th(s_\alpha(T_r(\alpha)))
        +e^{-\lambda (T_1(\alpha)-T_r(\alpha))}
        \Th(s_\alpha(T_1(\alpha)))
        \\
        &\leq
        \Th(s_-(\lambda,c))
        +\Th(rs_-(\lambda,c)),
    \end{align*}
    where the last the last right-hand side
    is independent of $\alpha$ and convergent
    to $0$ when $\lambda$ tends to $0$.
    Combining the latter inequalities,
    we finally obtain,
    \begin{equation}\label{Eq:K}
        \lambda
        e^{-\lambda T_r(\alpha)}
        \left(
        \mathcal{J}_0(s_0,c)
        -\mathcal{J}_0(s_0,\alpha)
        \right)
        \geq
        (1-3r)\Th_{0,\max}(c)-L(c)
        -\Th(s_-(\lambda,c))
        -\Th(rs_-(\lambda,c)),
    \end{equation}
    which is positive
    for $\lambda$ small enough,
    uniformly with respect to $\alpha$.
    This implies that
    $J(s_0,c)\geq V(s_0)$
    leading to 
    $J(s_0,c)= V(s_0)$
    and the fact that the constant
    control $c$ is optimal.
    Therefore,
    $(\lambda,c,\Vt,\M,\Th)$
    is indeed a
    monotonous reversed MFG travelling wave
    in the sense of Definition \ref{De:ReversedMFGTW}.
\end{proof}

\begin{remark}\label{Re:UniformityInL}
A key point in the proof of Theorem \ref{Th:Cooperation} is that Theorem \ref{Th:small_lambda} is in a sense ``uniform in $L$" in the following sense: consider $c$ fixed and, use the notation $s_-(L,\lambda)$ rather than $s_-(\lambda,c)$ to emphasise the dependence of $s_-$ on $L$. Assume that two Lagrangians $L_1$ and $L_2$ and two discount factors $\lambda_1\,, \lambda_2>0$ satisfy
\[L_1(c)=L_2(c)\text{ and } \lambda_1 L_1'(c)=\lambda_2L_2'(c).\] Then, clearly $\Th_{c,\lambda_1,L_1}=\Th_{c,\lambda_2,L_2}$ and $\Tht_{c,\lambda_1,L_1}=\Tht_{c,\lambda_2,L_2}$ so that $s_-(L_1,\lambda_1)=s_-(L_2,\lambda_2)$. In particular, if $(r,\lambda_1)$ are chosen so as to satisfy \eqref{Eq:Yu}, and such that the right hand-side of \eqref{Eq:K} is positive for $L=L_1$, then the same conditions are met by $(r,\lambda_2)$ and $L_2$. In particular, the optimality of $\overline \alpha\equiv c$ for the triplet $(c,\lambda_1,L_1)$ implies the optimality of $\overline \alpha\equiv c$ for $(c_,\lambda_2,L_2)$.
\end{remark}

\section{Proof of Theorem~\ref{Th:Cooperation}}
\begin{proof}[Proof of Theorem \ref{Th:Cooperation}]
    \emph{First step: define $\alpha_{\rm co}$.}
    We first need to introduce the Lagrangian
    \[ L_1:\alpha\mapsto\frac{\underline\eta}2 |\alpha|^2,\]
    where $\underline\eta$ is the unique critical point of $f$ in $[0;\eta]$.
    As $L_1(1)=\frac{\underline \eta}2$,
    \eqref{Eq:Tau0} ensures that $c_{\max}(L_2)>1$.
    Fix $c=1$,
    by Theorem~\ref{Th:small_lambda}, there exists $\lambda_0>0$
    such that there exists a monotonous reversed MFG travelling wave
    $(c,\lambda_0,\Th,\M,\cV)$
    with $\supp\M=[0,s_1]$.
    We are now in position to define $\alpha_{\rm co}$ as
    \[ \alpha_{\rm co}:(t,x)\mapsto \frac{x}{2s_1+t}.\]
    This implies that the solution of $\dot{x}_t=\alpha_{\rm co}(t,x_t)$
    starting from $x_0$ is affine and given by
    \[x_t=x_0+\frac{x_0}{2s_1}t.\]

    \emph{Second step: prove that the fishes are invading with $\alpha_{\rm co}$.}
    We define $(m,\th)$ as the solutions of
\[
    \begin{cases}\partial_tm
        +\partial_x(\alpha_{\rm co}m)
        =
        0
\\
        m(0,\cdot)= \M,
        \\
       \partial_t\th
        -\partial^2_{xx}\th
        =
        f(\th)
        -m\th,
 
 \\       \th(0,\cdot)= \Th.\end{cases}
\]    Observe that
    \begin{equation*}
        m(t,x)
        =
        \frac{2s_1}{2s_1+t}
        \M\left(\frac{2s_1+t}{2s_1}x\right).
    \end{equation*}
For any $\e>0$ and any $t\geq T_1(\e):=\frac{2s_1(\|\M\|_{L^{\infty}(\R)}-\e)}{\e}$,
    we have
    \[\|m(t,\cdot)\|_{L^{\infty}(R)}\leq\e.\]
    This implies that on $[T_1(\e),\infty)$,
    $\th$ satisfies
    \begin{equation*}
        \partial_t\th
        -\partial^2_{xx}\th
        \geq
        f(\th)
        -\e\th,
        \;\text{ with }\;
        \lim_{x\to -\infty}\th(T_1(\e),x)=0
        \;\text{ and }\;
        \lim_{x\to \infty}\th(T_1(\e),x)=1
    \end{equation*}
    Consider the non-linearity $f_{\e}:[0,1]\ni u\mapsto f(u)-\e u$.
    For $\e>0$ small enough, $f_\e$ is a bistable non linearity,
    with $\int_0^1 \tilde f_\e>0$;
    we call $(1-\delta_{\e})$ the largest root of $f_\e$ in $(0;1)$,
    it converges to $1$ as $\e$ tends to $0$.
    Using the maximum principle,
    the asymptotic behaviour of $\th(T_1(\e),\cdot)$
    and Theorem~\ref{Th:WithoutFisherman},
    we deduce that for $\e>0$ small enough 
    \[ \inf_{x\in \R_+}\lim_{t\to \infty}\theta(t,x)\geq r_\e,\] where $r_\e\to 1$ as $\e\to 0$. 
    so that,
    there exists $T_2(\e)>0$ such that,
    for any $t\geq T_2(\e)$, 
    \begin{equation}
        \label{Eq:SurSu}
        \theta(t,\cdot)
        \geq
        1-2\delta_{\e}
        \text{ in }
        \supp(m(t))=[0,s_1+\textstyle{\frac{t}2}],
    \end{equation}
    for some $\delta$ that will be chosen later.

    \emph{Third step: let $\lambda\to0$ with an appropriate Lagrangian.}
    We just constructed $\alpha_{\rm co}$ producing $\th_{\rm co}$
    such that the fishes are invading after some time $T_2(\e,\delta)$ which can be long.
    Since, we are interested in the large time regime,
    we will let $\lambda$ tends to $0$,
    with $c=1$ and $\th$ fixed.
    Recall that $\th'(0)=\lambda_0 L_1'(1)$,
    the only way to reduce $\lambda$ without
    changing $\th$ is by changing the Lagrangian into
    $L_q$ defined by
    \[
        L_q:\alpha\mapsto
        \frac{\underline\eta}2|\alpha|^{2q}
        \;\text{ with }\;
        q=\frac{\lambda_0}{\lambda},
    \]
For any $\lambda \in (0;\lambda_0)$, we have 
\[L_{\frac{\lambda_0}\lambda}(1)=L_1(1)\,, \lambda L_{\frac{\lambda_0}\lambda}'(1)=\lambda_0 L_1'(1).\]
We deduce that $\Th_{c=1,\lambda,L_{\frac{\lambda_0}\lambda}}=\Th_{c=1,\lambda_0,L_1}$ for any $\lambda\in (0;\lambda_0)$. We thus drop the underscript $(c,\lambda,L)$.  Remark~\ref{Re:UniformityInL} implies that $(c=1,\lambda,\Th,\M)$ is a reversed MFG travelling wave for the Lagrangian $L_{\frac{\lambda_0}\lambda}$ for any $\lambda \in (0;\lambda_0)$.
   
    Defining $\cV_{\lambda}$ as in \eqref{Eq:DefVSV},
    we deduce that
    $(1,\lambda,\cV_{\lambda},\M,\Th)$
    is a reversed MFG travelling wave
    as well.
    On the one hand,
    recall that for $x\in[0,s_1]=\supp(\M)$,
    using $L_{q}(1)=\frac{\underline{\eta}}2$,
    we have
    \begin{equation*}
        \cV_{\lambda}(x)
        =
        \lambda^{-1}(\Th(x)
        -\textstyle{\frac{\underline{\eta}}2})
        \leq
        \lambda^{-1}(\Th(s_1)
        -\textstyle{\frac{\underline{\eta}}2}).
    \end{equation*}
    On the other hand,
    letting $\lambda$ tends to zero,
    we have
    \begin{equation*}
        J(\alpha_{\rm co},x,\th)
        \geq
        \lambda^{-1}(1-2\delta)
        +\psi_{\lambda},
        \:\text{ with }\;
        \psi_{\lambda}
        =
        o(\lambda^{-1}).
    \end{equation*}
    Therefore, to conclude, it is sufficient
    to take $\e$ small enough so that
    $(1-2\delta_{\e})>\Th(s_1)$
    and to take $\lambda$ small enough so that
    $\lambda\psi_{\lambda}<1-2\delta_{\e}-\Th(s_1)+\frac{\underline{\eta}}2$.
\end{proof}

\section{Proof of Theorem \ref{Th:exist_nonmono}}
As we outlined in the introduction, the proof of Theorem \ref{Th:exist_nonmono} is very similar to that of Theorem \ref{Th:Main}, so that we will be content with highlighting the main differences, all of them stemming from the non-monotonicity of the profile $\Th$.

We fix, for the rest of this section, an integer $k\in \N^*$ and a couple $(\Th,\M)$ provided by Proposition \ref{Prop:exist_th_multiple_bumps}, such that $\Th$ has $k$ local maxima and $\M$ has an interval as a support. Up to a translation, $\supp(\M)=[0;s_1]$. We begin with two preliminary results.
\paragraph{Control of the Lipschitz constant of the value function} We start off with an estimate of the Lipschitz constant of the value function $\cV$ (which we recall is the value function expressed in similarity variables).

\begin{lemma}
    \label{lem:V_Bernstein}
    For  any $\lambda>0$ and $c\in (0;c_k(\lambda))$, let $(\Th,\M)$ be provided by Proposition \ref{Prop:exist_th_multiple_bumps}.    Let $\cV$ be the associated value function. Then, $\cV$ is Lipschitz continuous a Lipschitz constant bounded from above  by $\lambda^{-1}\|\Th'\|_{\infty}$.
\end{lemma}

\begin{proof}
    For any $\e_1\,, \e_2>0$, define $H_{\e_2}$ by
    \begin{equation*}
        H_{\e_2}(p)
        =
        \rho_{\e_2}\star H,
    \end{equation*}
    where $\rho_{\e_2}(p)=\e_2^{-1}\rho(\frac{p}{\e_2})$
    and $\rho$ a smooth even function whose support
    is included in $[-1,1]$, and let $\cV_{\e_1,\e_2}$ be the unique solution of
    \begin{equation}
        \label{eq:V_e12}
        \lambda \cV_{\e_1,\e_2}
        -\e_1 \cV_{\e_1,\e_2}''
        +c\cV_{\e_1,\e_2}'
        -H_{\e_2}(\cV_{\e_1,\e_2}')
        =
        \Th.
    \end{equation} Observe that as $H$ is Lipschitz continuous, $H_{\e_2}$ is Lipschitz continuous, with a Lipschitz constant uniform in $\e_2\to 0$.

    \begin{enumerate}\item\underline{ $L^{\infty}$-bound on $\cV_{\e_1,\e_2}$.}
    Let us rewrite \eqref{eq:V_e12} as
    \begin{equation*}
        \lambda \cV_{\e_1,\e_2}
        -\e_1 \cV_{\e_1,\e_2}''
        +\left(c
        -\int_0^1
        H'_{\e_2}(s\cV_{\e_1,\e_2}')ds
        \right)
        \cV_{\e_1,\e_2}'
        =
        \Th
        +H_{\e_2}(0).
    \end{equation*}
    Since $H$ is nonnegative, so is $H_{\e_2}$.
    Moreover, recall that $H(0)=0$,
    so $H_{\e_2}(0)\leq 1$ for any $\e_2$ small enough. Finally, recall that $0\leq \Th\leq 1$.    The maximum principle implies    \begin{equation*}
        0
        \leq
        \cV_{\e_1,\e_2}
        \leq
        2\lambda^{-1}.
    \end{equation*}
\item\underline{$L^{\infty}$-bound on $\cV'_{\e_1,\e_2}$.}
    Let us differentiate \eqref{eq:V_e12} to get
    \begin{equation*}
        \lambda \cV'_{\e_1,\e_2}
        -\e_1 \cV_{\e_1,\e_2}'''
        +(c
        -H'_{\e_2}(\cV_{\e_1,\e_2}'))
        \cV_{\e_1,\e_2}''
        =
        \Th'.
    \end{equation*}
    Since $H'_{\e_2}$ and $\Th'$ are H\"{o}lder continuous, standard elliptic estimates imply that
    $\cV_{\e_1,\e_2}$ is $C^3$. The maximum principle yields
    \begin{equation}
        \label{eq:V_e12'}
        \|\cV_{\e_1,\e_2}'\|_{\infty}
        \leq
        \lambda^{-1}\|\Th'\|_{\infty}.
    \end{equation}

\item\underline{Taking the limits $\e_2\to 0$, $\e_1\to 0$.}
    Recall that $H$ is Lipschitz continuous,
    so $H_{\e_2}$ is with a Lipschitz constant
    uniform with respect to $\e_2$.
From elliptic regularity, for any $\e_1\,, \e_2>0$, $\cV_{\e_1,\e_2}\in C^{2+\beta}$ for some $\beta>0$,
    uniformly with respect to $\e_2$.
    Therefore, we can extract a $\mathscr C^2_{\mathrm{loc}}$-converging subsequence as $\e_2\to 0$.    Taking the limit in \eqref{eq:V_e12}, it appears that the limit $\cV_{\e_1}$,
is the solution of    \begin{equation*}
        \lambda \cV_{\e_1}
        -\e_1 \cV_{\e_1}''
        +c\cV_{\e_1}'
        -H(\cV_{\e_1}')
        =
        \Th.
    \end{equation*}
From standard results in the theory of viscosity solutions,  $\cV_{\e_1}$ is  $\mathscr C^0_{\mathrm{loc}}$-convergent
    to $\cV$,    where $\cV$ is the unique viscosity solution to
    \begin{equation*}
        \lambda \cV
        +c\cV'
        -H(\cV')
        =
        \Th.
    \end{equation*}
    For any $s,r\in\R$,
    passing to the limit $\e_2\to 0$ in
    \eqref{eq:V_e12'}
    we finally obtain
    \begin{equation*}
        \cV_{\e_1}(s)
        -\cV_{\e_1}(r)
        -\lambda^{-1}\|\Th'\|_{\infty}
        |s-r|
        \leq
        0.
    \end{equation*}
Taking the limit $\e_1\to 0$ concludes the proof.        \end{enumerate}

\end{proof}

\paragraph{A convergence result as $c\to 0$} We now investigate the asymptotic behaviour of $\Gamma_{0,c}$ for small $c$.

\begin{lemma}
    \label{lem:Gamma_cvg}
    Take $k\in\N^*$
    and $c>0$.
    Up to a translation,
    we assume that
    $\Gamma_{0,c}(0)$ is the $k^{\text{th}}$
    local maximum of $\Gamma_{0,c}$. As $c\to 0$, $\Gamma_{0,c}$ $\mathscr C^1_{\mathrm{loc}}$-converges to $\Gamma_{0,0}$, where, up to a translation, $\Gamma_{0,0}$ reaches its global maximum at $s=0$.
    \end{lemma}
\begin{proof}[Proof of Lemma \ref{lem:Gamma_cvg}]
By continuity of the solutions of an ODE with respect to its parameters, it is enough to show (recall \eqref{Eq:Gamma0Max}) that\begin{equation}\Gamma_{0,\max}(c)\underset{c\to 0}\to \Gamma_{0,\max}(0).\end{equation} However, adapting the arguments of Lemma \ref{Le:Galere1} we see that $c\mapsto \Gamma_{0,c}(0)$ is non-decreasing. In particular, we deduce that $\Gamma_{0,c}(0)\to \ell_0$ as $c\to 0$, where $\ell_0>0$.  From Schauder estimates, $\Gamma_{0,c}$ converges in $\mathscr C^1_{\mathrm{loc}}(\R)$ to a solution $\gamma_0$ of the same equation as $\Gamma_{0,0}$, with $\gamma_0'(0)=0\,, \gamma_0(0)>0$. We deduce that $\gamma_0=\Gamma_{0,0}$ and that $\ell_0=\Gamma_{0,\max}(0)$.

\end{proof}

We are now in a position to prove Theorem \ref{Th:exist_nonmono}.

  \begin{proof}[Proof of Theorem \ref{Th:exist_nonmono}] 
   Most of the arguments in the present proof
    will be repeated from the proof of Theorem \ref{Th:Main}. We are thus satisfied with highlighting the salient differences.
    
    We let $\lambda>0$ and $c\in(0,c_k(\lambda))$
    and we consider the couple $(\Th,\M)$ provided
    by Proposition \ref{Prop:exist_th_multiple_bumps}.
    Up to a translation, we assume that
    $\inf(\supp(\M))=0$.
    From Lemma \ref{lem:V_Bernstein},   there exists $\underline{c}<0$,
    such that an optimal control $\alpha^*$
    has to satisfies the a priori estimates
    $\alpha^*\geq \underline{c}$.
  Using arguments akin to those employed in the course of Theorem \ref{Th:Main}    \begin{equation}
        \underline{c}
        \leq
        \alpha
        <
        c,
    \end{equation}
    and that satisfy $T_1(\alpha)<\infty$,
    where $T_1(\alpha)$, $s_-(\lambda,c)$
    and $\Tht$ are defined as in the proof
    of Theorem \ref{Th:Main}.
    Define $\overline{T}(\alpha)$ by
    \begin{equation*}
        \overline{T}(\alpha)
        =
        \inf\left\{t\geq0,
        \Th(s_{\alpha}(t))=\frac{\Gamma_{0,\max}(0)}2\right\}.
    \end{equation*}
We obtain, as in Theorem \ref{Th:Main},    \begin{multline*}
        \cVt(s^{\alpha}(\overline{T}(\alpha)))
        -J(s^{\alpha}(\overline{T}(\alpha)),\alpha)
        \geq
        e^{-\lambda(T_1(\alpha)-\overline{T}(\alpha))}
        \Bigl(
        \cVt(s^{\alpha}(T_{1}(\alpha)))
        -J(s^{\alpha}(T_{1}(\alpha)),\alpha)
        \Bigr)
        \\
        +\int_{\overline{T}(\alpha)}^{T_1(\alpha)}
        e^{-\lambda (t-\overline{T}(\alpha))}
        \Bigl(\Tht(s^{\alpha}(t))
        -\Th(s^{\alpha}(t))
        \Bigr)\,dt.
    \end{multline*}
    Then, observe that
    \begin{equation*}
        \cVt(s_{\alpha}(T_1))
        \geq
        -\lambda^{-1}L(c)
        \,\text{ and }\,
        J(s_{\alpha}(T_1))
        \leq
        \lambda^{-1}.
    \end{equation*}
    This implies
    \begin{equation*}
        e^{\lambda \overline{T}(\alpha)}
        (J(s_0,c)
        -J(s_0,\alpha))
        \geq
        -\lambda^{-1}(1+L(c))e^{-\lambda(T_1-\overline{T})}
        +\int_{\overline{T}(\alpha)}^{T_1(\alpha)}
        e^{-\lambda (t-\overline{T}(\alpha))}
        \Bigl(\Tht(s^{\alpha}(t))
        -\Th(s^{\alpha}(t))
        \Bigr)\,dt.
    \end{equation*}
    Observe that
    \begin{equation*}
        T_1-\overline{T}
        \geq
        \frac{\Gamma_{0,\max}(0)
             -2s_-(\lambda,c)}{2(c-\underline{c})},
    \end{equation*}
    with $\lim_{c\to0}s_-(\lambda,c)=-\infty$
    using Lemma \ref{lem:Gamma_cvg}.
    This implies that
    \begin{align*}
        e^{\lambda \overline{T}}
        (J(s_0,c)
        -J(s_0,\alpha))
        &\geq
        -\lambda^{-1}(1+L(c))e^{-\lambda(T_1-\overline{T})}
        +\int_{\overline{T}}^{\overline{T}+1}
        e^{-\lambda (t-\overline{T}(\alpha))}
        \Bigl(\Tht(s^{\alpha}(t))
        -\Th(s^{\alpha}(t))
        \Bigr)\,dt
        \\
        &\geq
        -\lambda^{-1}(1+L(c))e^{-\lambda(T_1-\overline{T})}
        +e^{-\lambda}\frac{\Gamma_{0,\max}(0)}4,
    \end{align*}
    where the last line is obtained for $c$ small enough.
    It is then straightforward to conclude.

\end{proof}

\section{Conclusion and open problems}
\label{sec:conclusion}
In this work, we have proposed a MFG-TW approach to the tragedy of the commons, and we would like to comment on one aspect, before moving on to open questions and further comments: throughout the article we have considered a bistable nonlinearity, that is used to model the (strong) Allee effect. Other nonlinearities are of paramount importance in spatial ecology, most notable monostable non-linearities. A first comment we would like to address is that of the possibility to extend our results to this setting (or obstructions to doing so). If we consider a monostable nonlinearity $f:[0;1]\to \R$ with $f(0)=f(1)=0\,, f>0\text{ in }(0;1)$ and such that $f'(0)>0$, no reversed traveling wave can be built, as can easily be infered from the associated phase portrait. However, if $f'(0)=0$ (such a case is often referred to as \emph{degenerate monostable}, and is used to model the weak Allee effect), a careful study of the phase portrait show that it is possible to build reversed travelling waves, and conclusions similar to that of Theorem \ref{Th:Quadratique} can be obtained. Overall,  it thus appear that, at least at a paradigmatic level, \emph{it is the combination of both the Allee effect and of te harvesting game structure that drives the population to extinction}.

On the same note, observe that the conclusions of Theorem \ref{Th:Quadratique} are valid whatever the precise shape of the non-linearity is, as long as a reversed travelling wave with a bounded second-order derivative can be built (which of course is not the case for non-degenerate monostable non-linearities, as explained above). 

Several questions that are as crucial from the applications' perspectives as they seem out of reach from the mathematical point of view remain omen. The three main of these questions are the following:
\begin{enumerate}
\item The first one is the \emph{stability} of the reversed MFG travelling waves we obtained. Most proofs of the stability of travelling waves for bistable propagation fronts rely on a specific structure of the PDE (typically, a variational structure, see \cite{Gallay_2007,Risler_2008} for a general approach. We also refer to \cite[Chapter 5]{Volpert_1994}). Here, the approach one should adopt is unclear, but the question of knowing whether or not reversed MFG travelling waves or (locally or globally) stable is quite important.
\item Let us stress a point related to the tragedy of the commons: we exhibited a particular coordinated strategy that globally outperforms the optimal MFG one while ensuring survival of the fishes' population. However, this was proved for a specific choice of Lagrangian. A natural question is: is it in general true that for a fixed Lagrangian $L$ and for a small enough discount factor $\lambda>0$ there always exists a coordinated strategy that does not extinguish the fishes' population while guaranteeing that each individual player actually harvests more fishes than he would in a competitive setting? We believe the answer to this question is affirmative, but a proof is at this point out of reach. 
\item Finally, let us comment on government regulations on fishing. Typically, governments enforce regulations on the number of catches to avoid overfishing, and they can also prohibit harvesting in certain protected zones. This leads to two qualitatively important questions: first, which regulations should a government impose so as to rule out these reversed MFG traveling waves? Second, in the construction of our reversed MFG traveling waves, each fisherman is fishing in the zone of transition between 0 and 1, and, were they to fish at the back of the front (where $\Th\approx 1$), they ould harvest as much fish, while not causing the extinction of the fishes. Thus, what are the best strategies, in terms of designing ``harvest-free" zones?\end{enumerate}

\bibliographystyle{alpha}
\bibliography{BiblioNash}

\end{document}